\documentclass[11pt, reqno]{amsart}
\usepackage{amssymb}
\usepackage{amsmath}
\usepackage{mathrsfs}
\usepackage{amssymb, amsmath, amsfonts}
\usepackage{bm}
\usepackage{bbm}
\usepackage{mathabx}
\usepackage{extarrows}
\usepackage{color}

\usepackage[hidelinks]{hyperref}

\def\tank#1{\protected@xdef\@thanks{\@thanks
 \protect\footnotetext[0]{#1}}}
\def\bigfoot{

 \@footnotetext}
\makeatother

\newcommand{\ea}{\end{array}}

\allowdisplaybreaks
\numberwithin{equation}{section}
\newtheorem{theorem}{Theorem}[section]
\newtheorem{lemma}{Lemma}[section]
\newtheorem{proposition}{Proposition}[section]

\newtheorem{remark}{Remark}[section]

\newtheorem{definition}{Definition}[section]

\def\beq{\begin{equation}}
\def\nneq{\end{equation}}

\def\bthm{\begin{theorem}}
\def\nthm{\end{theorem}}

\def\blem{\begin{lemma}}
\def\nlem{\end{lemma}}
\def\bprf{\begin{proof}}
\def\nprf{\end{proof}}
\def\bprop{\begin{prop}}
\def\nprop{\end{prop}}
\def\brmk{\begin{rem}}
\def\nrmk{\end{rem}}
  
\def\bexa{\begin{exa}}
\def\nexa{\end{exa}}
\def\bcor{\begin{cor}}
\def\ncor{\end{cor}}

\def\e{{\varepsilon}}

\def\Var{\mathrm {Var}}

\def\e{\mathrm {e}}
\def\card{{\mathrm{card}}}

\title[Lower functions and Chung's LILs of GFBM]{Lower functions and Chung's LILs of the 
generalized fractional Brownian motion}
  
\author{Ran Wang}
\address[]{Ran Wang, School of Mathematics and Statistics,  Wuhan University,  Wuhan, 430072,  
China.}
\email{rwang@whu.edu.cn}

\author{Yimin Xiao}
\address[ ]{Yimin Xiao, Department of Statistics and Probability, Michigan State University, East Lansing, 
MI 48824, USA.}\email{xiaoy@msu.edu}

\date{}

\begin{document}
\maketitle

 \noindent {\bf Abstract:}   Let $X:=\{X(t)\}_{t\ge0}$  be a 
  generalized fractional Brownian motion (GFBM) introduced by Pang and Taqqu (2019):  
$$
 \big\{X(t)\big\}_{t\ge0}\overset{d}{=}\left\{  \int_{\mathbb R}  \left((t-u)_+^{\alpha}-(-u)_+^{\alpha} \right) 
 |u|^{-\gamma} B(du)  \right\}_{t\ge0},
$$  with parameters $\gamma \in (0, 1/2)$ and 
$\alpha\in \left(-\frac12+ \gamma , \,   \frac12+ \gamma \right)$.   
  
Continuing the studies of sample path properties of GFBM $X$ in Ichiba, Pang and Taqqu (2021) 
and Wang and Xiao (2021), we establish integral criteria for the lower functions  of  $X$ at $t=0$ 
and at infinity by modifying the arguments of Talagrand (1996). As a consequence of the 
integral criteria, we derive the Chung-type laws of the iterated logarithm of $X$ at the  $t=0$
and at infinity, respectively. This solves a problem in Wang and Xiao (2021).
 
 \vskip0.3cm
 \noindent{\bf Keyword:} {Generalized fractional Brownian motion; lower function; Chung's LIL;
 small ball probability.}
 \vskip0.3cm

\noindent {\bf MSC: } {60G15, 60G17, 60G18, 60G22.}
\vskip0.3cm

 \section{Introduction}

The generalized fractional Brownian motion (GFBM, in  short) $X:=\big\{X(t)\big\}_{t\ge0}$ is a centered Gaussian  
self-similar process introduced by Pang and Taqqu \cite{PT2019} as the scaling limit of a sequence of 
power-law shot noise processes. It has the following integral representation:
 \begin{align}\label{eq X}
 \big\{X(t)\big\}_{t\ge0}\overset{d}{=}&\left\{  \int_{\mathbb R}  \left((t-u)_+^{\alpha}-(-u)_+^{\alpha} \right) 
 |u|^{-\gamma} B(du)  \right\}_{t\ge0},
 \end{align}
 where the parameters $\gamma$ and $\alpha$ satisfy
 \begin{align}\label{eq constant}
 \gamma\in\left(0,\ \frac12\right),  \ \  \alpha\in \left(-\frac12+\gamma, \  \frac12+\gamma \right),
 \end{align}
and  $B$ is a two-sided  Brownian motion on $\mathbb R$.
It follows that the Gaussian process $X$ is self-similar with index $H$ given by  
 \begin{align}\label{eq H}
 H=\alpha-\gamma+\frac12\in(0,1).
 \end{align}
     
If $\gamma=0$,  then $X$ becomes an ordinary fractional Brownian motion (FBM, in short)  $B^H$, which 
can be represented as:
\begin{align}\label{eq FBM}
\big\{B^H(t)\big\}_{t\ge 0}\overset{d}{=}\left\{ \int_{\mathbb R}  
\Big((t-u)_+^{H-\frac12}-(-u)_+^{H-\frac12} \Big) B(du)  \right\}_{t\ge0}.
\end{align}

As shown by Pang and Taqqu \cite{PT2019},  GFBM $X$ preserves the self-similarity property while the 
factor $|u|^{-\gamma}$ introduces non-stationarity of  the increments, which is useful for reflecting the 
non-stationarity increments property in physical systems. Ichiba, Pang and Taqqu \cite{IPT2020} established
 the H\"older continuity, the functional and local laws of the 
iterated logarithm of GFBM and showed that these properties are determined by the self-similarity index $H 
= \alpha-\gamma+ 1/2$.  More recently, Ichiba, Pang and Taqqu \cite{IPT2020b} studied the semimartingale 
properties of GFBM $X$ and its mixtures and applied them to model the volatility processes in finance.

In \cite{WX2021},  we studied some precise sample path properties of GFBM $X$, including the exact 
uniform modulus of  continuity, small ball probabilities, and Chung's LIL at any fixed point $t>0$. In contrast 
to the theorems of Ichiba, Pang and Taqqu \cite{IPT2020b}, our results show that the uniform modulus of  
continuity and Chung's LIL at any fixed point $t>0$ are determined mainly by the parameter $\alpha$, while
$\gamma$ plays a less important role. Roughly speaking, for $\alpha<1/2$, the results in \cite{WX2021} on 
uniform modulus of  continuity and  Chung's LIL at $t>0$ are analogous to the corresponding results for a
fractional Brownian motion with index $\alpha+1/2$.   
For example, Theorem 1.5 in \cite{WX2021} shows the following Chung's LILs for GFBM $X$ and  its derivative 
$X'$ (which exists when $\alpha>1/2$) at any fixed $t>0$:
\begin{itemize}
\item[(a).]\,
If $\alpha\in(-1/2+\gamma/2,1/2)$, then there exists a constant $c_{1,1}\in(0,\infty)$ 
 such that for every $t>0$, 
    \begin{align}\label{eq LIL01}
   \liminf_{r\rightarrow0+} \sup_{ |h|\le r}\frac{ |X(t+h)-X(t)|} {r^{\alpha+1/2}/(\ln \ln r^{-1})^{\alpha+1/2}}
   =c_{1,1}t^{-\gamma},   \ \ \ \text{a.s.}
   \end{align} 
  \item[(b).]\, 
 If $\alpha\in(1/2, \, 1/2+\gamma/2)$,  then there exists a constant $c_{1,2}\in(0,\infty)$  
 such that  for every $t>0$,   
    \begin{align}\label{eq LIL01b}
   \liminf_{r\rightarrow0+} \sup_{ |h|\le r}\frac{ |X'(t+h)-X'(t)|} {r^{\alpha-1/2}/(\ln \ln r^{-1})^{\alpha-1/2}}
   =c_{1,2} t^{-\gamma}, \  \ \ \ \text{a.s.}
   \end{align}      
  \end{itemize}  
As $t\rightarrow 0+$, the terms on the right-hand sides of \eqref{eq LIL01} and \eqref{eq LIL01b} tend to 
$+\infty$. This suggests that the scaling functions on the left-hand sides of \eqref{eq LIL01}  and 
\eqref{eq LIL01b} are not optimal in the neighborhood of the origin. The problem on Chung's LIL for GFBM 
$X$ at $t=0$ was left open in \cite{WX2021}.

The main objective of the present paper is to establish Chung's LILs of GFBM $X$ at $t=0$ and at infinity. 
In fact, we will prove more precise results, namely, integral criteria for the lower functions of 
$M:=\big\{M(t)\big\}_{t\ge0}:=\big\{\sup_{0\le s\le t} |X(s)|\big\}_{t\ge0}$ at $t=0$ 
and at infinity, which imply the following Chung's  LILs of GFBM $X$.

\begin{theorem}\label{thm Chung origin} 
Let $X=\{X(t)\}_{t\ge0}$  be a GFBM with parameters $\alpha$ and $\gamma$. Suppose
 $\alpha\in (-1/2+\gamma, 1/2)$.
\begin{itemize}
\item[(a).]  There exists a positive constant $\kappa_1\in (0,\infty)$ such that 
\begin{equation}\label{eq 11}
\liminf_{t\rightarrow 0+}\sup_{0\le s\le t}\frac{ |X(s)|}{t^H /\left(\ln\ln t^{-1}\right)^{\alpha+1/2}}=\kappa_1, 
\ \ \ \hbox{ a.s.}
\end{equation} 

\item[(b).]  There exists a positive constant $\kappa_2\in (0,\infty)$ such that 
\begin{equation}\label{eq 21}
\liminf_{t\rightarrow\infty}\sup_{0\le s\le t}\frac{|X(s)|}{t^H /\left(\ln\ln t\right)^{\alpha+1/2}}=\kappa_2, 
\ \ \ a.s.
\end{equation}  
\end{itemize}
 \end{theorem}
   
Similarly to the theorems of Ichiba, Pang and Taqqu \cite{IPT2020b} mentioned above, the self-similarity index 
$H$ plays an essential role in (\ref{eq 11}) and  (\ref{eq 21}). The results in \cite{IPT2020b,WX2021} and the 
present paper show that GFBM $X$ is an interesting example of self-similar Gaussian processes which has 
richer sample path properties than the ordinary FBM and its close relatives such as the Riemann-Liouville 
FBM (cf. e.g., \cite{CLRS11,ElN2011}), bifractional Brownian motion (cf. \cite{HV03,   LeiN09, RussoTudor, TX2007}), 
and the sub-fractional Brownian motion (cf. \cite{BGT04, ElN2012, Tudor07,YanShen10}). In this sense, GFBM
is a good object (see also \cite{TT18} for other examples related to stochastic partial differential 
equations driven by a fractional-colored Gaussian noise) that can be studied for the purpose to develop 
a general theoretical framework for studying the fine properties of all (or at least a wide class of) self-similar 
Gaussian processes which, to the best of our knowledge, is still not complete yet.

In the literature, limit theorems of the forms (\ref{eq 11})  and  (\ref{eq 21}) are also called ``the other law of 
the iterated logarithm" and there has been a long history of studying them. 
Chung \cite{Chung}  proved that if $S_n = 
\eta_1 +\cdots +\eta_n$, where $\{\eta_k\}$ is a sequence of i.i.d. random variables with 
mean 0, variance 1 and finite third moment, then  
 $$
 \liminf_{n\rightarrow\infty}\frac{\max_{1\le k\le n} |S_k|}{\sqrt{n/\ln\ln n}} =\frac{\pi}{\sqrt 8}, \ \ \ \text{a.s.}
 $$
 Chung \cite{Chung} also gave the corresponding large time result for Brownian motion.  The extra condition of 
 finite third moment on $\eta_1$ in \cite{Chung} was removed by Jain and Pruitt \cite{JP75}. There have been many 
 extensions of these results. For example, Cs\'aki  gave a converse of lower/upper class in \cite{Cs1978}, and he 
 found an interesting connection between Chung's ``other law" for Brownian motion and Strassen's LIL in \cite{Cs1980}. 
Kuelbs et al \cite{KLT1994} studied Chung's functional LIL for Banach space-valued Gaussian random vectors. Their 
results are applicable to Brownian motion and  provide interesting refinements to those in \cite{Cs1980}. Monrad 
and Rootz\'en \cite{MR1995}  proved Chung's LIL for a large class of Gaussian processes that have the property 
of strong local nondeterminism. Li and Shao \cite {LiShao01} and Xiao \cite{Xiao1997} extended the Chung's LIL in 
\cite{MR1995} to Gaussian random fields with stationary increments.  Chung's LILs have also been studied for 
non-Gaussian processes, we refer to Buchmann and Maller \cite{BM} and the references therein for more information.
  
\begin{remark}  {\rm The following are some remarks about Chung's LILs in Theorem  \ref{thm Chung origin}.

(i). Notice that the cases of $\alpha = 1/2$ and $\alpha\in (1/2, 1/2+\gamma)$ are excluded in Theorem 
\ref{thm Chung origin}. In the first case, the sample functions of $X$ are not differentiable, while in the second 
case the sample functions of $X$ are differentiable on $(0, \infty)$. In both cases, we have not be able to 
solve the problem whether  (\ref{eq 11})  and  (\ref{eq 21}) hold or not because the optimal small ball 
probability estimates for $\max_{t\in[0, 1]} |X(t)|$ have not been established for GFBM $X$ yet. See
 \cite[Remark 5.1]{WX2021} for more information. 

It is worth mentioning that, when $\alpha\in (1/2, 1/2+\gamma)$, $X$ has a modification that is continuously 
differentiable and its derivative $X'$ is a self-similar process with index $H' = \alpha-\gamma- 1/2 < 0$. The 
exact uniform modulus of continuity on $[a, b]$ with $0<a <b<\infty$ and Chung's LIL at any fixed $t>0$ have 
been proved for $X'$ in \cite{WX2021}. However, the asymptotic properties of $X'$ at $t = 0$ or at infinity 
have not been studied. Since $H' <0$, it is expected that $X'(t) \to \infty$ as $t \to 0+$ and $X'(t) \to 0$ as 
$t \to \infty$. It would be interesting from the viewpoint of the aforementioned general theoretical framework 
for self-similar Gaussian processes to prove both ordinary LIL and Chung's LIL of $X'$ at $t= 0$ and at $\infty$.

(ii). To prove Theorem  \ref{thm Chung origin}, we modify the argument of Talagrand \cite{Tal96}, which is 
concerned with the lower functions of FBM at $\infty$, to establish integral criteria for the lower functions of 
GFBM $X$ at $t= 0$ and at infinity. We remark that, in \cite{ElN2011, ElN2012}, El-Nouty has extended 
Talagrand's result to the Riemann-Liouville FBM and the sub-FBM to characterize their lower functions at 
$\infty$. 

For studying the lower functions of $M=\big\{\sup_{0\le s\le t} |X(s)|\big\}_{t\ge0}$ at $t = 0$ in this paper, 
the main difficulty comes from the singularity of the second moment of the increment of $X$ at $t = 0$. 
To elaborate, we consider a decomposition of GFBM:
 \begin{equation}\label{eq decom}
 \begin{split}
  X(t) &=\int_{-\infty}^0  \big((t-u)^{\alpha}-(-u)^{\alpha} \big) (-u)^{-\gamma} B(du)  +\int_0^t  (t-u)^{\alpha}   
  u^{-\gamma} B(du)\\
  &=: Y(t) + Z(t).
\end{split}
\end{equation}
The Gaussian processes   $Y=\big\{Y(t)\big\}_{t\ge0}$ and  $Z=\big\{Z(t)\big\}_{ t\ge0}$ are independent. The 
process $Z$ in \eqref{eq decom} is called a {\it generalized Riemann-Liouville FBM}, using the terminology 
of Ichiba, Pang and Taqqu \cite{IPT2020}.  By \cite[Lemmas 2.1 and 3.1]{WX2021}, 
there exist positive constants $c_{1, i}$, $i=3,\cdots, 6$, such that   for all $0 < s < t$, 
\begin{align}\label{Eq: Ymoment1}
c_{1,3}\frac{|t-s|^2}{t^{2-2H}}\le     \mathbb E\left[\big(Y(t)-Y(s) \big)^2\right]\le c_{1,4}\frac{|t-s|^2}{s^{2-2H}}.
 \end{align}  
 and 
 \begin{align}\label{Eq: Zmoment1}
c_{1,5}\frac{|t-s|^{2\alpha+1}}{t^{2\gamma}}\le     \mathbb E\left[\big(Z(t)-Z(s) \big)^2\right]
\le c_{1,6}\frac{|t-s|^{2\alpha+1}}{s^{2\gamma}}.
 \end{align} 
These bounds are optimal when $s \le t \le c_{1,7}s$ for any constant $c_{1,7}>1$. Therefore, 
the small values of $s$ have some effects on $Y$ and $Z$ when $H<1$ and $\gamma>0$, 
respectively. 
The singularity at $s=0$ brings two technical difficulties when we modify the approach in Talagrand 
\cite{Tal96}: one is in   estimating the metric entropy 
 for proving Proposition \ref{prop mtu1}; the other is in constructing of the sequences in Section 
 \ref{subsect tn} in order to use the (generalized) Borel-Cantelli lemma.

(iii).  From the proofs of Theorems \ref{thm Chung origin} and \ref{thm main 0}, we can 
verify that the conclusions of Theorem \ref{thm Chung origin} also hold for the generalized 
Riemann-Liouville FBM $Z$ in \eqref{eq decom}. For example,   when $\alpha\in (-1/2+\gamma, 1/2)$,
\begin{equation}\label{eq Z11}
\liminf_{t\rightarrow 0+}\sup_{0\le s\le t}\frac{ |Z(s)|}{t^H /\left(\ln\ln t^{-1}\right)^{\alpha+1/2}}=c_{1,8}\in (0,\infty) 
\ \ \ a.s.
\end{equation}  
For the process $Y$  defined in \eqref{eq decom},  
when $\alpha\in (-1/2+\gamma, 1/2)$
  we obtain  that by \eqref{eq 11} and \eqref{eq Z11},    
\begin{equation}\label{eq Y11}
\liminf_{t\rightarrow 0+}\sup_{0\le s\le t}\frac{ |Y(s)|}{t^H /\left(\ln\ln t^{-1}\right)^{\alpha+1/2}}=c_{1,9}\in [0,\infty)
\ \ \ a.s.
\end{equation}    
Since an optimal upper bound for the small ball probability estimates has not been established 
for $Y$  yet (see  \cite[Lemma 7.1]{WX2021}),  we are not able to decide if $c_{1,9} > 0$ or $c_{1,9} = 0$. 
 
 }
\end{remark}

The rest of this paper is organized as follows. In Section 2, we state Theorems \ref{thm main 0} and 
\ref{thm main infty} which provide integral criteria for the lower functions of $M$ at $t=0$ and at infinity.  
From these integral criteria, we derive Chung's LILs for $X$ in Theorem  \ref{thm Chung origin}. 
 
Section \ref{sec:prelim} contains some preliminary results. In Section 4 and Section 5, we prove 
Theorems \ref{thm main 0} and \ref{thm main infty}, respectively.   
  
\section{Main results and Proof of Theorem  \ref{thm Chung origin} }

The following definition of the lower classes for the process $M=\big\{M(t)\big\}_{t\ge0}$ is adapted from  
\cite{Revesz}. This book  provides a systematic and extensive account on the studies of lower 
and upper classes for Brownian motion, random walks and their functionals.
 
 \begin{definition}\label{Def:LL}
 \begin{itemize}
 \item[(a).] A function $f(t), t>0$, belongs to the lower-lower class of the process $M$ at  $\infty$ 
 (resp. at $0$), denoted by $f\in LLC_{\infty}(M)$ (resp. $f\in LLC_{0}(M)$), if for almost all 
 $\omega\in \Omega$ there exists $t_0=t_0(\omega)$ such that $M(t)\ge f(t)$ for every $t>t_0$ 
 (resp. $t<t_0$).
 
\item[(b).]\, A function $f(t), t>0$, belongs to the lower-upper class of the process $M$ at  $\infty$ 
(resp. at $0$), denoted by $f\in LUC_{\infty}(M)$ (resp. $f\in LUC_{0}(M)$), if for  almost all $\omega\in 
\Omega$ there exists a sequence $0<t_1(\omega)<t_2(\omega)<\cdots<$ with $t_n(\omega)\uparrow\infty$ 
(resp.  $t_1(\omega) > t_2(\omega)>\cdots>$ with $t_n(\omega)\downarrow0$), as $n\rightarrow \infty$, 
such that $M(t_n(\omega)) \le f(t_n(\omega)), n\in \mathbb N$. 
\end{itemize}
\end{definition}

Since in the present paper, $M(t)=\sup_{0\le s\le t} |X(s)|$, we will also write the lower classes in 
Definition \ref{Def:LL} as $LLC_{0}(X)$ and $LUC_{0}(X)$ and call a function $f \in LLC_{0}(M)$ 
(resp. $f \in LUC_{0}(M)$) a lower-lower (resp. lower-upper) function of $X$ at $t = 0$. It is known from 
Talagrand \cite{Tal96} that small ball probability estimates are essential for studying the lower classes of a 
stochastic process.  

By the self-similarity of GFBM $X$, we have
\begin{equation}\label{eq M1}
\mathbb P\big(M(t)\le \theta t^H \big)=\mathbb P\big(M(1)\le \theta\big)=:\varphi(\theta).
\end{equation}
Wang and Xiao \cite{WX2021} proved the following small ball probability estimates for GFBM: 
If $\alpha\in (-1/2+\gamma, \,1/2)$, then there exist    constants $\kappa_3>
 \kappa_4>0$ such that for all $t>0$ and $ 0<\theta<1$, 
 \begin{align*}\label{eq small X}
\exp\bigg(- \kappa_{3}  \Big(\frac{t^H}{\theta}\Big)^{\frac{1}{\beta}} \bigg)\le 
 \mathbb P\bigg\{\sup_{s\in [0,t]} |X(s)|\le  \theta \bigg\} 
 \le  \exp\bigg(- \kappa_4 \Big(\frac{t^H}{\theta}\Big)^{\frac{1}{\beta}} \bigg).
 \end{align*} 
Here and in the sequel, $\beta =\alpha+1/2$. It follows that  for any $\theta\in(0,1)$, 
 \begin{equation}\label{eq varphi1}
\exp\Big(- \kappa_{3}  \theta^{-\frac{1}{\beta}} \Big)\le \varphi(\theta)\le 
\exp\Big(- \kappa_{4}  \theta^{-\frac{1}{\beta}} \Big).
\end{equation}
  
Now we state our first main result, which gives an integral criterion for the lower functions of 
GFBM $X$ at $t = 0$.  It shows that, besides $\beta =\alpha+1/2$, the self-similarity index 
$H = \alpha-\gamma+\frac12$ plays an essential role. 
\begin{theorem}\label{thm main 0} Assume $\alpha\in (-1/2+\gamma, \,1/2)$.   
Let $\xi:(0,\e^{-\e}]
\rightarrow (0,\infty)$ be a nondecreasing continuous function.
\begin{itemize}
\item[(a)] (Sufficiency).  If   
\begin{equation}\label{eq 01}
\frac{\xi(t)}{t^H}\  \text{ is bounded  and}\  I_0(\xi):= \int_0^{\e^{-\e}}  \left(\frac{\xi(t)}{t^H}\right)^{-1/\beta} 
\varphi\left(\frac{\xi(t)}{t^H}\right)\frac{dt}{t}<+\infty,
\end{equation}
then $\xi\in LLC_0(X)$.
\item[(b)](Necessity). Conversely, if  $\frac{\xi(t)}{t^{(1+\varepsilon_0)H}}$ is non-increasing for some constant 
$\varepsilon_0>0$ and if  $\xi\in LLC_0(X)$,  then \eqref{eq 01} holds.     
\end{itemize}
\end{theorem}

Theorem \ref{thm main infty} is an analogous result for GFBM $X$ at infinity.

\begin{theorem}\label{thm main infty} Assume $\alpha\in (-1/2+\gamma, \,1/2)$.   
Let $\xi:[\e^{\e},\infty) \rightarrow (0,\infty)$ be a nondecreasing continuous function. Then  
$\xi\in LLC_{\infty}(X)$ if and only if 
\begin{equation}\label{eq infty}
\frac{\xi(t)}{t^H}\  \text{ is bounded  and}\  I_{\infty}(\xi):=\int_{\e^\e}^{\infty}  \left(\frac{\xi(t)}{t^H}\right)^{-1/\beta} 
\varphi\left(\frac{\xi(t)}{t^H}\right)\frac{dt}{t}<+\infty.
\end{equation}
\end{theorem}

The proofs of Theorems \ref{thm main 0} and  \ref{thm main infty} will be given in Sections 
4 and 5 below. 
First, let us apply them to prove Theorem   \ref{thm Chung origin}.

 We will make use of the following zero-one laws at $t=0$ and $ \infty$. Eq.  \eqref{Eq:Chung01law} 
 follows from \cite[Proposition 3.3]{TX2007} which provides a zero-one law for the lower class of 
 (not necessarily Gaussian) self-similar processes with ergodic scaling transformations. Recall from 
 \cite{Taka89,TX2007} that for every $a> 0$, $a\ne 1$, the scaling transformation $S_a$ of $X$ is 
 defined by $S_a(X) = \{a^{-H} X(at)\}_{t \ge 0}$. Notice that  for any $H$-self-similar process $X$, 
 the   scaling transformation $S_{ a}$ preserves the distribution of $X$. Hence the notions of 
 ergodicity and mixing of $S_{ a}$ can be defined in the usual way, cf. Cornfeld et al. \cite{CFS}.
 Following Takashima \cite{Taka89}, we say that an $H$-self-similar process $X =
\{X(t)\}_{ t \ge 0}$ is ergodic (resp. strong mixing) if for every $a > 0, a \ne 1$, the scaling transformation 
$S_{ a}$ is ergodic (resp. strong mixing). This, in turn, is equivalent to saying that all the shift 
transformations for the corresponding stationary process obtained via Lamperti's transformation 
$L(X)= \{ e^{-H t} X(e^t)\}_{ t \in \mathbb R}$ are ergodic (resp. strong mixing).
  
For GFBM $X$ in (\ref{eq X}), in order to verify the ergodicity of its scaling  transformations, we 
use the representation (\ref{eq X}) to show that the autocovariance function of $L(X)$ 
satisfies $e^{-H t} \mathbb E\big(X(e^t)X(1)\big) = O(e^{-\kappa_5 t}) $ as $t \to \infty$, where 
$\kappa_5 = \min\{\frac 1 2 - \gamma, \, \frac 1 2 + \gamma - \alpha\} > 0$.
By the Fourier inversion formula,  $L(X)$ has a continuous spectral density function. It follows 
from \cite[Theorem 8]{Ma} that $L(X)$ is strong mixing and thus is ergodic. Hence, 
 \cite[Proposition 3.3]{TX2007} is applicable to GFBM $X$ and \eqref{Eq:Chung01law} follows. 
 
The proof of (\ref{Eq:Chung01law2}) is similar to that of \cite[Proposition 3.3]{TX2007} 
with minor modifications. See \cite{Taka89} for more zero-one laws for self-similar processes with 
ergodic scaling transformations.

 \begin{lemma}\label{lem:01lawZ} 
 Assume $\alpha\in(-1/2+\gamma,\,1/2)$. There exist constants $c_{2,1}, c_{2,1}' \in[0,\infty]$
 such that
\begin{equation}\label{Eq:Chung01law} 
 \liminf_{t\rightarrow0+}  \sup_{0\le s\le t} \frac{ |X(s)|} {t^{H}/(\ln\ln t^{-1})^{\beta}}
 =c_{2,1},\ \ \  \ \ \ \text{a.s.}
\end{equation}
and 
\begin{equation}\label{Eq:Chung01law2} 
 \liminf_{ t\rightarrow\infty}  \sup_{0\le s\le t} \frac{ |X(s)|} {t^{H}/(\ln\ln t)^{\beta}}
 =c_{2,1}',\ \ \  \ \ \ \text{a.s.}
\end{equation}
\end{lemma}

For a constant $\lambda>0$, let $f_\lambda$ be the function defined by 
\begin{align}
f_\lambda(t):=     \frac{\lambda t^H}{( \ln|\ln t|)^{\beta}}, \ \ \ t>0.
\end{align}

\begin{proof}[Proof of Theorem \ref{thm Chung origin}]
When  $\alpha\in(-1/2+\gamma,\,1/2)$, by combining Theorem  \ref{thm main 0}  (resp. Theorem  \ref{thm main infty})
with (\ref{eq varphi1}),  we derive that if $\lambda<\kappa_4^{\beta}$, then $f_{\lambda}\in LLC_{0}(X)$ 
(resp. $f_{\lambda}\in LLC_{\infty}(X)$), 
else if $\lambda>\kappa_3^{\beta}$, then $f_{\lambda}\in LUC_{0}(X)$ (resp. $f_{\lambda}\in LUC_{\infty}(X)$).  
These, together with the zero-one law in Lemma \ref{lem:01lawZ}, imply the desirable results in Theorem 
 \ref{thm Chung origin}. 
 \end{proof}

\section{Some preliminary results }\label{sec:prelim}

In this section, we provide some preliminary results that will be useful for proving Theorems 
\ref{thm main 0} and  \ref{thm main infty}. Their proofs are modifications of those in Talagrand \cite{Tal96}.

\begin{lemma}   
If $\alpha\in (-1/2+\gamma, \,1/2)$, then there exists a constant $c_{3,1}>0$ such that for 
all  $0<t<u$, $\theta, \eta>0$,
\begin{equation}\label{eq Mtu}
\mathbb P \big(M(t)\le \theta t^H,\, M(u)\le \eta\big)\le   2\varphi(\theta)\exp\bigg(-\frac{u-t}{c_{3,1}  \,
u^{\frac{\gamma}{\beta}}\eta^{\frac{1}{\beta}}}\bigg).
\end{equation}
\end{lemma}

\begin{proof}   
If $(u-t)/\big(u^{\frac{\gamma}{\beta}}\eta^{\frac{1}{\beta}}\big)\le 2$,  then it is obvious that \eqref{eq Mtu}   
holds with $c_{3,1}=2/\ln 2$. Hence, we only need to prove \eqref{eq Mtu} in the case of 
$(u-t)/\big(u^{\frac{\gamma}{\beta}}\eta^{\frac{1}{\beta}}\big)>2$.
The proof is divided into two steps.

{\bf Step 1.}  We define an increasing sequence $\{t_n\}_{n\ge0}$ as follows.  Set $t_0=t$. For any 
$n\ge1$, if $t_{n-1}$ has been defined,  then we choose $t_n>t_{n-1}$  such that 
\begin{align*}
t_n -t_{n}^{\frac{\gamma}{\beta}} \eta^{\frac{1}{\beta}}=t_{n-1}.
\end{align*} 
Consider the event
\begin{align*}
A_k:=\big\{M(t)\le \theta t^H \big\}\cap \big\{M({t_k}) \le \eta\big\}.
\end{align*}
It suffices to prove that for any $k\ge1$, we have 
\begin{align}\label{Eq: PAk}
\mathbb P(A_k)\le \varphi(\theta) \rho^{k}, 
\end{align}
where $\rho\in (0,1)$ is a constant that depends on $\beta$ and $\gamma$ only. Indeed, if $k$ is the 
largest integer such that $t_k\le u$, then $k+1\ge (u-t)/\big({u^{\frac{\gamma}{\beta}} \eta^{\frac{1}{\beta}}}\big)$, 
which implies  $k\ge (u-t)/\big({2u^{\frac{\gamma}{\beta}} \eta^{\frac{1}{\beta}}}\big)$ and 
\begin{align*}
A_k\supset \big\{M(t)\le \theta t^H \big\}\cap \big\{M(u)\le \eta\big\}.
\end{align*}
Hence, \eqref{Eq: PAk} implies \eqref{eq Mtu}.

{\bf Step 2.}  We prove \eqref{Eq: PAk} by induction over $k$. The result holds for $k=0$ by \eqref{eq M1}.  
For the induction step, we observe that 
$$
A_{k+1}\subset A_k  \cap \big\{|U|\le 2\eta \big\},
$$
where $U:= X({t_{k+1}})-X({t_k})$.  By using \eqref{eq X}, $U$ can be rewritten as follows $U=U_1+U_2$, 
where 
\begin{align*}
U_1&:=\int_{t_k}^{t_{k+1}} (t_{k+1}-u)^{\alpha}u^{-\gamma}B(du), \\
 U_2&:=\int_{-\infty}^{t_{k}} \big[(t_{k+1}-u)^{\alpha}- (t_{k}-u)^{\alpha}   \big]|u|^{-\gamma}B(du).
\end{align*}
Notice that $U_1$ is a Gaussian random variable with 
  $$
\mathbb E \big[U_1\big]=0 \ \  \text{ and }\ \ \Var\big(U_1\big)\ge \frac{1}
{t_{k+1}^{2\gamma}}|t_{k+1}-t_k|^{2\beta}=\eta^{2}.
$$
Thus, we have
$$ \mathbb P\big(|U_1|\le  2\eta\big)\le \Phi(2)-\Phi(-2),
$$
where $\Phi$ denotes the distribution function of   a standard Gaussian random variable.  
Consequently, by Anderson's inequality \cite{And1955} and  the independence of  $U_1$ and 
$\sigma\big\{B(s); s\le t_k\big\}$, we have 
\begin{align*}
 \mathbb P(A_{k+1})\le&\,  \mathbb E\Big[
\mathbb P \Big(A_{k} \cap\big\{|U_1+U_2|\le  2\eta\big\}\,\big|\,\sigma\big\{ B(s);   s\le t_k\big\}\Big)\Big]\\
= &\,  \mathbb E\Big[ {\mathbbm 1}_{A_{k}} \cdot
\mathbb P\big( |U_1+U_2|\le  2\eta\,\big|\,\sigma\big\{B(s);  s\le t_k\big\}\big)\Big]\\
\le & \, \mathbb P\big(A_{k}\big)\cdot\mathbb P\big(|U_1|\le  2\eta\big)\\
\le &\,  \mathbb P\big(A_{k}\big)\cdot \big( \Phi(2)-\Phi(-2)\big).
\end{align*}
Therefore,  we have proved \eqref{Eq: PAk} with $\rho=\Phi(2)-\Phi(-2)$. This finishes the proof of 
(\ref{eq Mtu}). \end{proof}

Set 
\begin{align}
\psi(\theta):=-\log \varphi(\theta).
\end{align}
Then $\psi$ is positive and non-increasing.  According to Borell \cite{Borell1974}, we know that 
$\psi$ is convex which implies the existence of the right derivative $\psi'$ of $\psi$. Thus, 
$\psi'\le 0$ and $|\psi'|$ is non-increasing.  

By the small ball probability estimates in (\ref{eq varphi1}), we see that there exists a constant $K_1\ge1$  
such that for all $\theta<1$, 
\begin{align}\label{eq sbp}
\frac{1}{K_1 \theta^{1/\beta}}\le \psi(\theta)\le \frac{K_1}{\theta^{1/\beta}}.
\end{align}

The following lemmas give more properties of the functions $\varphi$ and $\psi$,
which are similar to those in  Talagrand \cite[Section 2]{Tal96}. 
\begin{lemma}\label{lem psi'}
There exists a constant $K_2\ge \max\big\{ 2^{1+1/\beta},\, 2\left(2K_1^2\right)^{\alpha}
K_1 \big\}$ such that for all $\theta\in (0, 1/K_2)$, 
\begin{align}\label{eq psi'}
 - \frac{K_2}{\theta^{1+1/{\beta}}}\le  \psi'(\theta) \le  -\frac{ 1}{K_2\theta^{1+1/{\beta}}}.
\end{align}
\end{lemma}

\begin{lemma}\label{lem dom} 
There exists a constant  $K_3\ge K_2 2^{1+1/\beta}$ such that for all 
$ \theta \le \varepsilon\le 2\theta<1$, 
\begin{align}\label{eq control}
\exp\left(-K_3 \frac{|\varepsilon-\theta|}{\theta^{1+1/\beta}}\right)\le \frac{\varphi(\varepsilon)}{\varphi(\theta)}
\le \exp\left(K_3 \frac{|\varepsilon-\theta|}{\theta^{1+1/\beta}}\right).
\end{align}
\end{lemma}

\begin{lemma}\label{lem increase} For all $\theta<\theta_0:=(\beta/K_2)^{\beta}$,   the function 
$\theta^{-1/\beta}\varphi(\theta)$ is increasing. 
\end{lemma}
Since the proofs of the above three lemmas are similar to the analogous lemmas for FBM in 
Talagrand \cite[Section 2]{Tal96}, we only prove  Lemma \ref{lem increase} as an example.
\begin{proof}[Proof of Lemma \ref{lem increase}] The right derivative of  the function 
$\theta^{-1/\beta}\varphi(\theta)$ is 
$$
\left(-\frac{1}{\beta \theta} -\psi'(\theta)\right)\theta^{-1/\beta}\varphi(\theta),
$$
which is positive in the interval $(0, \theta_0)$ for the positive constant $\theta_0=(\beta/K_2)^{\beta}$ by \eqref{eq psi'}. 
The proof is complete.     
\end{proof}

Given a Gaussian process $\{X(t)\}_{t\in \mathbb T}$, let us denote by $N(\mathbb T, \, d_X, \, \varepsilon)$ 
the smallest number of open balls of radius $\varepsilon$ for the canonical distance $d_X(s,t)=\|X(s)-X(t)\|_2$ 
that are needed to cover $\mathbb T$, where $\|\cdot\|_2$ is the $L^2(\mathbb P)$-norm. Recall the following 
lemma from Talagrand \cite{Tal96}.
\begin{lemma}\cite[Lemma 2.1]{Tal96}\label{lem Dudley}
There exists a  universal constant $K_4$ such that for all $t_0\in \mathbb T$ and $x>0$, 
\begin{align}
\mathbb P\left(\sup_{t\in \mathbb T}\left|X(t)-X(t_0)\right| \ge K_4 x\int_0^{\infty} \sqrt{\log N(\mathbb T,\, d_X,\, 
\varepsilon) }d\varepsilon \right)\le \exp\big(-x^2\big). 
\end{align}
\end{lemma}

The following proposition will be important for proving the necessity in Theorems \ref{thm main 0} and  
\ref{thm main infty}. 

\begin{proposition}\label{prop mtu1} 
Assume $\tau=\min\left\{\frac14(1-H),\, \frac{H}{4},\, \frac14\left(\frac12-\gamma\right)\right\}$. Then for any    
$u>t>0$, $\theta>0, \eta>0$,  we have  
  \begin{equation}\label{eq mtu1}
  \begin{split}
 & \mathbb P\Big( \left\{M(t)\le \theta t^H  \right\}\cap \left\{ M(u)\le \eta u^H \right\}\Big)\\
  &\le \exp\left[-\frac{1}{K}\left(\frac{u}{t}\right)^{\tau}\right]+\varphi(\theta)\varphi(\eta) 
  \exp\left[K  \left(\frac{u}{t}\right)^{-\tau}\left(\theta^{-1-1/\beta}+\eta^{-1-1/\beta} \right)\right].
  \end{split}
      \end{equation}
     \end{proposition}  
\begin{proof} 
 We set $v:=\sqrt{ut}$. The idea is that if $t\ll u$, then $t\ll v\ll u$. We recall the  stochastic integral 
 representation in \eqref{eq X}, and define
 $$G(s, x):= \left((s-x)_+^{\alpha}-(-x)_+^{\alpha} \right) |x|^{-\gamma}, \ \ \ \text{for } s, x\in \mathbb R.
$$
Consider the following two processes:
\begin{align}\label{eq X12}
X_1(s):=\int_{|x|\le v} G(s,x)  B(dx),\, \,\,\,\,
X_2(s):=\int_{|x|> v} G(s,x)  B(dx).
\end{align}
Thus, $X(s)=X_1(s)+X_2(s)$ and the processes $X_1$ and $X_2$ are independent.

For any $\delta>0$,   we have
\begin{equation}\label{eq X12-0}
\begin{split}
 & \mathbb P\Big( \left\{M(t)\le \theta t^H  \right\}\cap \left\{ M(u)\le \eta u^H \right\}\Big)\\
= &\, \mathbb P\left(\sup_{0\le s\le t} |X(s)|\le \theta t^H,\, \sup_{0\le s\le u}  |X(s)|\le \eta u^H  \right)\\
\le  &\, \mathbb P\left(\sup_{0\le s\le t}  |X_1(s)|\le \left(\theta+\delta\right) t^H,  \, \sup_{0\le s\le u}  |X_2(s)|
\le \left(\eta+\delta\right) u^H \right)\\
&\,+  \mathbb P\left(\sup_{0\le s\le t}  |X_2(s)|\ge  \delta  t^H   \right)+\mathbb P\left( \sup_{0\le s\le u}  |X_1(s)|
\ge  \delta u^H \right).
\end{split}
\end{equation}
By the independence of $X_1$ and $X_2$,  we have
\begin{equation}\label{eq X12-1}
\begin{split}
&  \mathbb P\left(\sup_{0\le s\le t}  |X_1(s)|\le \left(\theta+\delta\right) t^H,  \, \sup_{0\le s\le u}  |X_2(s)|
\le \left(\eta+\delta\right) u^H \right)\\ 
  =&\, \mathbb P\left(\sup_{0\le s\le t}  |X_1(s)|\le \left(\theta+\delta\right) t^H\right)\cdot 
  \mathbb P\left(\sup_{0\le s\le u}  |X_2(s)|\le \left(\eta+\delta\right) u^H \right).
\end{split}
\end{equation}
Notice that
\begin{equation}\label{eq X12-2}
\begin{split}
 & \mathbb P\left(\sup_{0\le s\le t}  |X_1(s)|\le \left(\theta+\delta\right) t^H\right)\\
 \le &\,  \mathbb P\left(\sup_{0\le s\le t}  |X(s)|\le \left(\theta+2\delta\right) t^H\right)+
  \mathbb P\left(\sup_{0\le s\le t}  |X_2(s)|\ge  \delta  t^H\right)
\end{split}
\end{equation}
and 
\begin{equation}\label{eq X12-3}
\begin{split}
 & \mathbb P\left(\sup_{0\le s\le u}  |X_2(s)|\le \left(\theta+\delta\right) u^H\right)\\
 \le &\,  \mathbb P\left(\sup_{0\le s\le u}  |X(s)|\le \left(\theta+2\delta\right) u^H\right)+
  \mathbb P\left(\sup_{0\le s\le u}  |X_1(s)|\ge  \delta  u^H\right).
\end{split}
\end{equation}
Plugging \eqref{eq X12-1} to \eqref{eq X12-3} into \eqref{eq X12-0}, we get 
\begin{equation}\label{eq X12-4}
\begin{split}
 & \mathbb P\Big( \left\{M(t)\le \theta t^H  \right\}\cap \left\{ M(u)\le \eta u^H \right\}\Big)\\
 \le &\, \varphi(\theta+2\delta)\varphi(\eta+2\delta) + 2 \mathbb P\left(\sup_{0\le s\le t}  |X_2(s)|
 \ge  \delta  t^H\right)
  + 2\mathbb P\left(\sup_{0\le s\le u}  |X_1(s)|\ge  \delta  u^H\right).
\end{split}
\end{equation}
By the convexity of $\psi$ and \eqref{eq psi'}, we have 
\begin{align*}
\psi(\theta+2\delta)\ge &\, \psi(\theta)+2\delta\psi'(\theta)
\ge \, \psi(\theta)-\frac{2\delta K_2}{\theta^{1+1/\beta}}.
\end{align*}
Hence, we have
$$
\varphi(\theta+2\delta)\le \varphi(\theta)\exp\left(\frac{2\delta K_2}{\theta^{1+1/\beta}}\right).
$$
This and a similar inequality for $\varphi(\eta+2\delta)$ show that it suffices to show that when 
$\delta=(t/u)^{\tau}$  with $\tau=\min\left\{\frac14(1-H), \, \frac{H}{4}, \,\frac14\left(\frac12-\gamma\right)\right\}$,  
the last two terms of  \eqref{eq X12-4} are bounded by $\exp\big(-\frac{1}{K}\big({u}/{t}\big)^{\tau}\big)$. 
This will be done  through the following two lemmas whose proofs are postponed. 

\begin{lemma}\label{lem diam} If $\alpha\in (-1/2+\gamma,1/2)$, then  there exist constants $c_{3,2}, 
c_{3,3}>0$ satisfying that
\begin{itemize}
\item[(a)] For any $0\le s \le u$, we have 
\begin{equation}\label{Eq:  B11}
\|X_1(s)\|_2\le c_{3,2} \left\{\begin{array}{ll}
 v^{H}, \ \quad &\hbox{ if }\ \beta\in \left( \gamma, \, \frac12\right);\\ 
v^{\frac{1 }{2}-\gamma} u^{\beta-\frac12},  \ \quad &\hbox{ if } \ \beta\in \left[\frac12 , 1+\gamma\right).
\end{array}
\right.
 \end{equation} 
\item[(b)] For any $0\le s\le t$,  we have 
\begin{equation}\label{Eq: B12}
\|X_2(s)\|_2\le c_{3,3} t v^{\beta-\gamma-1}.
\end{equation}
\end{itemize}
\end{lemma}  
  
Recall $X=Y+Z$  in \eqref{eq decom}.  Observing  the moment estimates of $Y$ and $Z$ in 
\eqref{Eq: Ymoment1} and \eqref{Eq: Zmoment1}, it is easy to see that for any $b>a>0$,  $\varepsilon>0$,
the covering numbers of $[a, b]$ in the canonical metrics of $Y$ and $Z$ satisfy
$$N\big([a,b], \, d_Y, \, \varepsilon\big) \le  c_{1,4}^{\frac12}(b-a) a^{H-1} \varepsilon^{-1}, \ \ \ N\big([a,b], \, 
d_Z, \, \varepsilon\big) \le  c_{1,6}^{\frac12} (b-a) a^{-\frac{\gamma}{\beta} } \varepsilon^{-\frac{1}{\beta}}.
$$ 
The next lemma gives their  estimates  when $a=0$.
\begin{lemma}\label{lem covering number}   There exist   constants $c_{3,4}, c_{3,5}>0$ satisfying that
for  any $b>0,\varepsilon>0$, 
\begin{align}\label{Eq cov Y}
N \big([0, b], d_{Y}, \varepsilon\big) \le c_{3,4} b^{H }\,  \varepsilon^{-1}
\end{align}
and 
\begin{align}\label{Eq cov Z}
N\big([0, b], d_{Z}, \varepsilon\big) \le c_{3,5} b^{\frac{H}{\beta}}\,  \varepsilon^{- \frac 1 {\beta}}.
\end{align}
\end{lemma}
 
Combining  Lemma \ref{lem diam} and Lemma \ref{lem covering number}, we derive that 
there exist constants $c_{3,6}, c_{3,7}>0$ satisfying that 
  \begin{equation}\label{Eq:  T1} 
\begin{split}
&\int_0^{\infty} \sqrt{\log N([0,u], d_{X,1},\varepsilon)}d\varepsilon \\
 \le&\,  c_{3,6} \left\{\begin{array}{ll}
 v^{H}\sqrt{\log \left(\frac u v\right) }, \ \quad &\hbox{ if }\ \beta\in \left( \gamma, \frac12\right);\\ 
 u^{H}\left(\frac{v}{u} \right)^{ 1/2-\gamma}\sqrt{\log \left(\frac uv\right)},  \ \quad &\hbox{ if } \
  \beta\in \left[ \frac12, 1+\gamma\right),
\end{array}
\right.
 \end{split}
 \end{equation} 
  and 
\begin{align}\label{Eq:  T2} 
\int_0^{\infty} \sqrt{\log N\big([0,t], d_{X,2},\varepsilon\big)}d\varepsilon \le c_{3,7}  \left(\frac{t}{v}\right)   
v^{H}\sqrt{\log  \left(\frac v t\right)}.
\end{align}
Since $t/v=v/u=\sqrt{t/u}$, the conclusion in Proposition \ref{prop mtu1}  follows from Lemma \ref{lem Dudley},
 \eqref{Eq:  T1}, \eqref{Eq:  T2}  and the choice of $\tau$.  

To prove \eqref{Eq:  T1} and \eqref{Eq:  T2}, we use the decomposition of $X=Y+Z$  in \eqref{eq decom}. 
Since $$N\big([0,u], d_{X,i},\varepsilon\big)\le N\big([0,u], d_{Y,i},\varepsilon/2\big)+N\big([0,u], 
d_{Z,i},\varepsilon/2\big), \ \ \ i=1,2,$$ 
 it suffices to prove  \eqref{Eq:  T1}  and \eqref{Eq:  T2}  for $Z$ and $Y$ separately.

If $\beta\in (\gamma, 1/2)$,  then 
\begin{equation*} 
\begin{split}
\int_0^{\infty} \sqrt{\log N\big([0,u], d_{Z,1},\varepsilon\big)}d\varepsilon&\,\le  \int_0^{c_{3,2} v^{H}}
\sqrt{\log \big(c_{3,5}u^{H/\beta}\varepsilon^{- 1/\beta}\big)  }d\varepsilon\\
&\, =  \beta u^{H}\int_{c_{3,2} ^{-1/\beta} (u/v)^{H/\beta} }^{\infty} 
\sqrt{\log {(c_{3,5}x)}} x^{-\beta-1}dx\\
&\,=2\beta c_{3,5}^{-1} u^H \int_{\sqrt{\log\left(c_{3,5}c_{3,2} ^{-1/\beta } 
(u/v)^{H/\beta}   \right)}}^{\infty} t^2 \e^{-\beta t^2}dt\\
&\,\le   c_{3,8}  v^{H}\sqrt{\log \left(\frac u v\right)},
\end{split}
\end{equation*}
for some $c_{3,8}>0$, here in the second and third steps,  the changes of variables $\varepsilon
= u^Hx^{-\beta}$ and $x=\e^{t^2}/c_{3,5}$ are used,  respectively,  and  in the last step we have 
used the following element inequality: there exists a 
constant $c_{3,9}>0$ satisfying that  for all $a$ large enough,
\begin{align}\label{eq ineq}
 \int_{a}^{\infty} t^2 \e^{-\beta t^2} dt\le c_{3,9} a \e^{-\beta a^2}.
\end{align}

Similarly, 
\begin{equation*} 
\begin{split}
\int_0^{\infty} \sqrt{\log N\big([0,u], d_{Y,1},\varepsilon\big)}d\varepsilon&\,\le  \int_0^{c_{3,2}  v^{H}}
\sqrt{\log \big(c_{3,4}u^{H }\varepsilon^{- 1}\big)  }d\varepsilon\\
&\, =  u^{H}\int_{c_{3,2} ^{-1 } (u/v)^{H }   }^{\infty} 
\sqrt{\log {(c_{3,4}x)}} x^{-2}dx\\
&\, = 2  c_{3,4}^{-1} u^H 
\int_{\sqrt{\log\left(c_{3,4}c_{3,2} ^{-1} (u/v)^{H} \right)}}^{\infty} t^2 \e^{- t^2}dt\\
&\,\le  c_{3,10}v^{H}\sqrt{\log \left(\frac u v\right)},
\end{split}
\end{equation*} 
for some constant $c_{3,10}>0$, here     the changes of variables
  $\varepsilon= u^Hx^{-1}$ and  $x =\e^{t^2}/c_{3,4}$ are used in the second and third steps, 
 respectively, and \eqref{eq ineq} is used in the last step.
   
By using the same  procedures, we  can prove the remainder   of \eqref{Eq:  T1}  and \eqref{Eq:  T2}.  
The details are omitted here. The proof of Proposition \ref{prop mtu1} is complete.
\end{proof}

\begin{proof}[Proof of Lemma \ref{lem diam}]
(a). If $\beta\in (\gamma, 1/2)$, then $\alpha\in (\gamma-1/2, 0)$. By using  the inequality $|(s+x)^{\alpha}-
x^{\alpha}|\le x^{\alpha}$ for any $s, x>0$, we have 
 \begin{align*}
\| X_1(s)\|_2^2=&\,  \int_{-v}^0 \left|(s-x)^{\alpha}-(-x)^{\alpha} \right|^2 |x|^{-2\gamma}dx+\int_0^{v\wedge s} 
(s-x)^{2\alpha} x^{-2\gamma}dx\\
= & \, \int_0^v  \left|(s+x)^{\alpha}-x^{\alpha} \right|^2 x^{-2\gamma}dx+
\int_0^{v\wedge s} (s-x)^{2\alpha} x^{-2\gamma}dx\\
 \le& \,  \int_0^v    x^{2\alpha-2\gamma}dx +\int_0^{v\wedge s} \left(v\wedge s-x \right)^{2\alpha} x^{-2\gamma}dx\\
 =& \frac{1}{2\alpha-2\gamma+1} v^{2\alpha-2\gamma+1}+  \mathcal B(2\alpha+1, 1-2\gamma)
 (v\wedge s)^{2\alpha-2\gamma+1}\\
 \le & \left(\frac{1}{2\alpha-2\gamma+1}+ \mathcal B(2\alpha+1, 1-2\gamma)\right) v^{2H}.
 \end{align*}

If $\beta\in [1/2, 1+\gamma)$, then  $\alpha\ge0$.  In this case, by using the inequality  
$\left|(s-x)^{\alpha}_+-(-x)_+^{\alpha}\right|\le  2 u^{\alpha}$ for all $s\in [0, u]$ and  $|x| \le u$,  we have
   \begin{align*}
\| X_1(s)\|_2^2\le \, 4 u^{2\alpha}\int_{|x|\le v} |x|^{-2\gamma}dx
=\,  \frac{4}{  1-2\gamma}   v^{1-2\gamma}u^{2\beta-1}.
\end{align*}

(b).  For any $0\le s\le t$,  by using the inequality $\left|(s+x)^{\alpha}-x^{\alpha}\right| \le |\alpha|x^{\alpha-1}s$ 
for any $\alpha\in (-1/2+\gamma, 1/2), s, x\ge0$, we have
 \begin{align*}
 \|X_2(s)\|_2^2=&\, \int_{|x|\ge v}\left((s-x)^{\alpha}_+-(-x)_+^{\alpha} \right)^2|x|^{-2\gamma}dx\\
 =&\, \int_{x\ge v}\left((s+x)^{\alpha}-x^{\alpha} \right)^2x^{-2\gamma}dx\\
 \le &\,  |\alpha|^2 s^2 \int_{x\ge v} x^{2\alpha-2\gamma-2}dx\\
 =&\, \frac{|\alpha|^2}{ 1+2\gamma-2\alpha } s^2 v^{2\beta-2\gamma-2}.
\end{align*}
This implies \eqref{Eq: B12}.
  \end{proof}  
 
\begin{proof}[Proof of Lemma \ref{lem covering number}] 
We use the argument in the proof of Lemma 4.1 in \cite{WX2021}. Since the proofs of \eqref{Eq cov Y} 
and \eqref{Eq cov Z} are similar, we only  prove \eqref{Eq cov Z} here. 

It follows from \eqref{Eq: Zmoment1} that  there exists a constant $c_{3,11}>0$ such that  for any $t>s>0$
 \begin{align}\label{eq d12-1}
d_{Z}(0,t)=\|Z(t)-Z(0)\|_2\le c_{3,11}|t|^H,
\end{align}
 and 
\begin{align}\label{eq d12-2}
d_{Z}(s,t)=\|Z(s)-Z(t)\|_2\le \frac{c_{3,11}}{s^{\gamma}}|t-s|^{\beta}.
\end{align}

Let $t_0=0, t_1 = \varepsilon^{1/H}$. For  any $n \ge 2$, if $t_{n-1}$ has been defined, we define
\begin{equation}\label{def:tn}
t_n = t_{n-1} + t_{n-1}^{\frac \gamma {\beta}} \varepsilon^{\frac 1 {\beta}}.
\end{equation}
It follows from \eqref{eq d12-2} that for all $n\ge2$,
\begin{equation*}  
\begin{split}
\|Z(t_n) - Z(t_{n-1})\|_2 &\le  \frac {c_{3,11}} {t_{n-1}^\gamma} \big| t_n - t_{n-1}\big|^{ \beta} 
\le  c_{3,11} \varepsilon.
\end{split}
\end{equation*}
Hence $d_{Z}(t_n,\, t_{n-1}) \le  \,c_{3,11} \varepsilon$ for all $n \ge 1$.

Since $[0, u]$ can be covered by the intervals $[t_{n-1}, \, t_n]$ for $n = 1, 2, \ldots, L_\varepsilon$, where 
$L_\varepsilon$ is the largest integer $n$ such that $t_n \le b$, we have $N ([0, b], d_Z, \varepsilon) 
\le L_\varepsilon+1 \le 2L_\varepsilon$.

In order to estimate $L_\varepsilon$, we write $t_n = a_n \varepsilon^{1/H}$ for all $n \ge 1$.
Then, by \eqref{def:tn}, we have $a_1=1$ and
\begin{equation*} 
a_n = a_{n-1} + a_{n-1}^{\frac \gamma {\beta}}, \quad \forall \, n \ge 2.
\end{equation*}
 Denote by $\rho  = \frac{H}{\beta}$. By (4.9) in \cite{WX2021}, we know  that there exist positive and finite 
constants $c_{3,12} \le  2^{-\gamma/(H\rho)}\rho^{1/\rho}$ and $c_{3,13} \ge 1$ such that 
\begin{equation}\label{Eq:an}
c_{3,12}\, n^{1/\rho} \le a_n \le c_{3,13}\, n^{1/\rho}, \qquad \forall \, n \ge 1.
\end{equation} 
By \eqref{Eq:an}, we have
\begin{equation*}\label{Eq:Lep}
L_\varepsilon = \max\left\{n; a_n \varepsilon^{\frac 1 H} \le b \right\} \le  c_{3,12}^{-\rho}\, b^{\rho}\,\varepsilon^{- \frac 1 {\beta}}.
\end{equation*}
This implies that for all   $\varepsilon \in (0, 1)$,
\begin{align*} 
N \big([0, b], d_{Z}, \varepsilon\big) \le 2 c_{3,12}^{-\rho}\,b^{\rho}\,  \varepsilon^{- \frac 1 {\beta}}.
\end{align*}  
The proof of Lemma \ref{lem covering number} is complete.
 \end{proof}

We end this section with the following lemma from Talagrand \cite{Tal96}, which will be used to prove 
the necessary parts in Theorems \ref{thm main 0} and \ref{thm main infty}. 

\begin{lemma}\cite[Corollary 2.3]{Tal96} \label{lem GBC}
Let $J \subseteq {\mathbb N}$. If there exist some positive numbers $K_5$ and $\varepsilon$ such 
that for all $i\in J$,
\begin{align}
\sum_{j<i}\mathbb P(A_i\cap A_j)\le \mathbb P(A_i)\bigg(K_5+(1+\varepsilon)\sum_{j<i}\mathbb P(A_j)\bigg),
\end{align}
and assume that $\sum_{i\in J}\mathbb P(A_i)\ge (1+2K_5)/\varepsilon$, then we have 
$\mathbb P\big(\bigcup_{i\in J}A_i\big)\ge (1+2\varepsilon)^{-1}$.
\end{lemma}

\section{Proof of Theorem \ref{thm main 0}}
\subsection{Sufficiency of the integral condition}

In this part, we always assume that $\xi(t)$ is a nondecreasing continuous   function such that $\xi(t)/t^H:(0,\e^{\e}]\rightarrow (0,\infty)$ is 
bounded and $I_0(\xi)<+\infty$. We   prove that $\xi(t)\le M(t)$ for all $t$ small enough in probability one.

Using the argument in the proof of \cite[Lemma 3.1]{Tal96},  we first prove the following result. 
\begin{lemma}\label{lem function10} Suppose that $\xi(t)$ is a nondecreasing continuous function such that 
$\xi(t)/t^H$ is bounded and $I_0(\xi)<+\infty$. Then
\begin{align}\label{eq f10}
\lim_{t\rightarrow0}\frac{\xi(t)}{t^H}=0.
\end{align}
\end{lemma}
\begin{proof} Recall the number $\theta_0$ in Lemma \ref{lem increase}. By the boundedness  and the 
continuity of $\xi(t)/t^H$, we know 
$$
c_{4,1}:=\inf\left\{\frac{\varphi(\theta)}{\theta^{1/\beta}}; \, \theta_0\le \theta\le \sup_{t>0} \frac{\xi(t)}{t^H} \right\}>0.
$$ 
For any $t\le u\le 2t$, we have
$$
\frac{\xi(u)}{u^H}\ge  \frac{\xi(t)}{(2t)^H}.
$$
By Lemma \ref{lem increase}, we know that the function $\varphi(\theta)/\theta^{1/\beta}$ is increasing   
in $(0,\theta_0)$. Thus for any $t\le u\le 2t$,
$$
\left(\frac{\xi(u)}{u^H}\right)^{-1/\beta} \varphi\left(\frac{\xi(u)}{u^H} \right)\ge \min\left\{
\left(\frac{\xi(t)}{(2t)^H}\right)^{-1/\beta} \varphi\left(\frac{\xi(t)}{(2t)^H} \right), \, c_{4,1}  \right\}.
$$
Consequently, we get 
$$
\int_t^{2t} 
\left(\frac{\xi(u)}{u^H}\right)^{-1/\beta} \varphi\left(\frac{\xi(u)}{u^H} \right)\frac{du}{u} \ge \log 2\cdot
 \min\left\{
\left(\frac{\xi(t)}{(2t)^H}\right)^{-1/\beta} \varphi\left(\frac{\xi(t)}{(2t)^H} \right),\, c_{4,1} \right\}.
$$
This, together with  the  fact that  $I_0(\xi)<+\infty$ and the monotonicity of $\varphi(\theta)/\theta^{1/\beta}$ in 
$(0,\theta_0)$, implies \eqref{eq f10}. The proof is complete.
 \end{proof}

In order to prove  the sufficiency in Theorem \ref{thm main 0}, we construct the sequences 
$\{t_n\}_{n\ge1}$,  $\{u_n\}_{n\ge1}$ and $\{v_n\}_{n\ge1}$ by recursion as follows. 
Let $L>2H$ be a  constant, whose value will be chosen later.   We start with $t_1=\e^{-\e}$. Having 
constructed $t_n$, we set
\begin{align} 
u_{n+1}:=&\inf\Big\{u<t_n; \, u+t_{n}^{\frac{\gamma}{\beta}}\xi(u)^{\frac1\beta}\ge t_n \Big\};\\
v_{n+1}:=&\, \inf\Bigg\{u<t_n;\, \xi(u)\bigg( 1+L\left(\frac{\xi(u)}{u^H} \right)^{1/\beta}\bigg)\ge \xi(t_n) \Bigg\}; 
\label{eq vn+11}\\
t_{n+1}:=&\, \max\{u_{n+1},\, v_{n+1}\}.\label{eq tn+110}
\end{align}
By the continuity of $\xi$, we have 
\begin{align}\label{eq un+11}
u_{n+1}=\, t_n-t_{n}^{\frac{\gamma}{\beta}}\xi(u_{n+1})^{\frac1\beta}.
 \end{align}
 
\begin{lemma}\label{lem 00}   The sequence $\{t_n\}_{n\ge1}$ is decreasing and 
\begin{equation}\label{eq tn0}
\lim_{n\rightarrow\infty} t_n=0.
\end{equation}
\end{lemma}
\begin{proof} By the construction of $\{t_n\}_{n\ge1}$,  it is obvious to see that $\{t_n\}_{n\ge1}$ is 
decreasing and $t_{\infty}:=\lim_{n\rightarrow\infty}t_n\ge0$ exists.   Next, we prove  $t_{\infty}=0$. 
Suppose, otherwise, that $t_n\ge t_{\infty}>0$ for all $n\ge1$.  By the continuity of $\xi$, we 
have  $ \lim_{n\rightarrow \infty}\xi(t_n)=\xi(t_{\infty})>0$.    By \eqref{eq un+11}, we have $t_n=v_n$ 
for $n$ large enough but then. Since $\xi$ is continuous, we have that  for all $n$ large enough,
$$
\xi(t_{n})=\xi(t_{n+1}) \Bigg( 1+L\bigg(\frac{\xi(t_{n+1})}{t_{n+1}^H} \bigg)^{1/\beta} \Bigg)
\ge\xi(t_{n+1}) \Bigg( 1+L\bigg(\frac{\xi(t_{\infty})}{t_1^H} \bigg)^{1/\beta} \Bigg),
$$
which  contradicts with the convergence of $\{\xi(t_n)\}_{n\ge1}$.  The proof is complete.
\end{proof}

\begin{lemma}\label{lem Mtn0}
If $M(t_n)\ge \xi(t_n)\Big(1+L\big(\frac{\xi(t_n) }{t_n^H}\big)^{1/\beta}\Big)$ for all $n\ge n_0$, then 
$M(t)\ge \xi(t)$ for   all $t\in (0,  t_{n_0}]$.
\end{lemma}
\begin{proof} Assume $n\ge n_0$ and 
$t_{n+1}< t\le t_{n}$.  By   \eqref{eq vn+11} and \eqref{eq tn+110}, we have
$$
\xi(t)\le \xi(t_n)\le  \xi(t_{n+1})  \Bigg( 1+L\bigg(\frac{\xi(t_{n+1})}{t_{n+1}^H} \bigg)^{1/\beta} \Bigg) 
\le M({t_{n+1}})\le M(t).
$$
Thus, $M(t)\ge \xi(t)$ for all $t\in (0,t_{n_0}]$, as desired. 
\end{proof}

By \eqref{eq M1}, we have 
\begin{equation}\label{Eq:Mtn}
\begin{split}
\, \mathbb P\Bigg(M(t_n)<\xi(t_n)\bigg(1+L\Big(\frac{\xi(t_n)}{t_n^H}\Big)^{1/\beta} \bigg) \Bigg)
=\,\varphi\Bigg(\frac{\xi(t_n)}{t_n^H}\bigg(1+L\Big(\frac{\xi(t_n)}{t_n^H}\Big)^{1/\beta}\bigg) \Bigg).
\end{split}
\end{equation}
By using the Borell-Cantell lemma and Lemma \ref{lem Mtn0},  the proof of the sufficiency  in  Theorem   
\ref{thm main 0} will be finished if we show that the terms in (\ref{Eq:Mtn}) form a convergent series.
 Applying \eqref{eq control} with 
$\theta= \xi(t_n)/t_n^H$, $\varepsilon=\theta+L\theta^{1+1/\beta}$ (since $\lim_{n\rightarrow\infty}
\xi(t_n)/t_n^H=0$) to see that it suffices to show that 
\begin{align}\label{eq finite}
\sum_{n=1}^{\infty}\varphi\left(\frac{\xi(t_n)}{t_n^H}\right)<+\infty.
\end{align}
This is done  by using the following lemma.
   
 \begin{lemma}\label{lem finite}
 \begin{itemize}
 \item[(a)] If $n$ is large enough and $t_{n}=u_{n}$, then we have 
 \begin{align}\label{eq tun1}
 \varphi\bigg(\frac{\xi(t_{n})}{t_{n}^H}\bigg)\le c_{4,2}\int_{t_{n}}^{t_{n-1}} 
 \left(\frac{\xi(t)}{t^H}\right)^{-1/\beta}\varphi\left(\frac{\xi(t)}{t^H}\right) \frac{dt}{t},
 \end{align}
 where $c_{4,2}$ depends on $\beta$ and $H$ only. 
 
 \item[(b)] If $n$ is large enough and $t_{n}=v_{n}$, then there exists  a  constant $L>2H$ depending on 
 $\beta$ and $H$ only such that 
 \begin{align}\label{eq tvn1}
      \varphi\bigg(\frac{\xi(t_{n})}{t_{n}^H}\bigg)\le \frac12 \varphi\left(\frac{\xi(t_{n-1})}{t_{n-1}^H}\right).
  \end{align}  
 \end{itemize}
 \end{lemma}
 
 We postpone the proof of Lemma \ref{lem finite}. Let us first finish the proof of \eqref{eq finite}, hence 
the sufficiency  in  Theorem   \ref{thm main 0}.  Consider the set   
 $$J:=\big\{n_k\in \mathbb N;\, t_{n_k}=u_{n_k}\ge v_{n_k}\big\}.$$ By \eqref{eq tun1}, we have 
 \begin{equation}\label{eq finite1}
 \begin{split}
 \sum_{n_k\in J}\varphi\left(\frac{\xi(t_{n_k})}{t_{n_k}^H}\right)\le&\,  c_{4,2}\sum_{n_k\in J}
 \int_{t_{n_k}}^{t_{n_{k}-1}}  \left(\frac{\xi(t)}{t^H}\right)^{-1/\beta}\varphi\left(\frac{\xi(t)}{t^H}\right) \frac{dt}{t}\\
 \le &\, c_{4,2} \int_0^{\e^{-\e}}  \left(\frac{\xi(t)}{t^H}\right)^{-1/\beta}\varphi\left(\frac{\xi(t)}{t^H}\right) \frac{dt}{t} 
 <\infty.
 \end{split}
 \end{equation}
 Let $n_{k-1}$ and $n_k$ be two consecutive terms of $J$. If there exists an integer $n$ such that 
 $n_{k-1}<n<n_{k}$, then set $p:=n-n_{k-1}$. Since  $t_{n}=v_{n}$   for all  $n\in \mathbb N\setminus J$, 
 Eq. \eqref{eq tvn1} implies that 
 $$
 \varphi\left(\frac{\xi(t_n)}{t_n^H}\right)\le 2^{-p}  \varphi\bigg(\frac{\xi(t_{n_{k-1}})}{t_{n_{k-1}}^H}\bigg).
 $$
 Thus, we obtain by setting $n_0=1$, 
\begin{equation*} 
 \begin{split}
 \sum_{n\in \mathbb N\setminus J} \varphi\left(\frac{\xi(t_n)}{t_n^H}\right)=&\, \sum_{k=1}^{\infty} 
\sum_{n_{k-1}<n< n_k}  \varphi\left(\frac{\xi(t_n)}{t_n^H}\right) 
\le   \,   \sum_{k=1}^{\infty}  \varphi\bigg(\frac{\xi(t_{n_{k-1}})}{t_{n_{k-1}}^H}\bigg)\\
\le &\, c_{4,2}\int_0^{\e^{-\e}} \left(\frac{\xi(t)}{t^H}\right)^{-1/\beta}\varphi\left(\frac{\xi(t)}{t^H}\right) \frac{dt}{t} 
<+\infty.
 \end{split}
 \end{equation*} 
This, together with \eqref{eq finite1},  implies \eqref{eq finite}. Therefore, the sufficiency  in Theorem \ref{thm main 0} 
 has been proved.
 \qed

 Next, we prove  Lemma \ref{lem finite}.
 
 \begin{proof}[Proof of Lemma \ref{lem finite}]
 (a) Assume  $t_{n}=u_{n}\ge v_{n}$. By \eqref{eq vn+11}, we have
 \begin{align}\label{eq tn+111}
 \xi(t_{n-1})\le \xi(t_{n})\Bigg( 1+L\bigg(\frac{\xi(t_{n})}{t_{n}^H} \bigg)^{1/\beta}\Bigg).
 \end{align}
By \eqref{eq un+11} and the monotonicities of $\varphi$ and $\xi$, we have 
\begin{align*}
   \int_{t_{n}}^{t_{n-1}}\left(\frac{\xi(t)}{t^H}\right)^{-1/\beta}\varphi\left(\frac{\xi(t)}{t^H}\right) \frac{dt}{t}
\ge&\,   
\left(t_{n-1}-t_{n}\right)t_{n-1}^{-\frac{\gamma}{\beta}}\xi\left(t_{n-1}\right)^{-1/\beta}\varphi\left(\frac{\xi(t_{n})}{t_{n-1}^H}\right)\\
=&\,  \left( \frac{\xi(t_{n})}{\xi(t_{n-1})}\right)^{1/\beta} \varphi\left(\frac{\xi(t_{n})}{t_{n-1}^H}\right)\\
\ge&\, 
\frac12\varphi\left(\frac{\xi(t_{n})}{t_{n-1}^H}\right),
\end{align*}
where in the last inequality, we have used the fact that  $\xi\left(t_{n}\right)\ge \xi\left(t_{n-1}\right)/2^{\beta}$ 
for all $n$ large enough   by   \eqref{eq tn+111}, Lemma \ref{lem function10} and Lemma \ref{lem 00}.  

Now,  using \eqref{eq un+11} again, we know that  for all 
 $n$ large enough,
 \begin{align*}
 \frac{\xi(t_{n})}{t_{n-1}^H}=&\, \frac{\xi(t_{n})}{t_{n}^H} \left(\frac{t_{n}}{t_{n-1}}\right)^H
 =\frac{\xi(t_{n})}{t_{n}^H}\Bigg(1- \bigg(\frac{\xi(t_{n})}
 {t_{n-1}^H}\bigg)^{1/\beta} \Bigg)^{H}\\
 \ge &\, \frac{\xi(t_{n})}{t_{n}^H}\Bigg(1-c_{4,3}\bigg(\frac{\xi(t_{n})}{t_{n}^H}\bigg)^{1/\beta} \Bigg). 
 \end{align*}
 Here $c_{4,3}>0$ which only depends on $H$. 
 Thus, \eqref{eq tun1} follows from \eqref{eq control}.
 
 (b). If $t_{n}=v_{n}\ge u_{n}$, then by the continuity of $\xi$ we have 
 \begin{align}\label{eq XLN}
 \xi(t_{n-1})=\xi(t_{n})\Bigg( 1+L\bigg(\frac{\xi(t_{n})}{t_{n}^H} \bigg)^{1/\beta}\Bigg).
  \end{align}
 Since $t_{n}\ge u_{n}$,  it follows from \eqref{eq un+11} and \eqref{eq XLN}  that  
  \begin{equation}\label{eq t n n+1}
 \begin{split}
 \frac{\xi(t_{n-1})}{t_{n-1}^H}=&\, \frac{\xi(t_{n-1})}{t_{n}^H}\left( \frac{t_{n}}{t_{n-1}}\right)^{H}\\
 \ge &\, \frac{\xi(t_{n})}{t_{n}^H}  \Bigg( 1+L\bigg(\frac{\xi(t_{n})}{t_{n}^H} \bigg)^{1/\beta}\Bigg) 
 \Bigg( 1-   \bigg(\frac{\xi(t_{n})}{t_{n-1}^H} \bigg)^{1/\beta}\Bigg)^H \\
    \ge &\, \frac{\xi(t_{n})}{t_{n}^H}  \Bigg( 1+(L-c_{4,4})\bigg(\frac{\xi(t_{n})}{t_{n}^H} \bigg)^{1/\beta}\Bigg),
     \end{split}
 \end{equation} 
 for some positive constant $c_{4,4}$ which only depends on $H$ and $\beta$. 
 
 Set $\theta=\xi(t_{n})/ t_{n}^{H}$, $\varepsilon=\theta+(L-c_{4,4})\theta^{1+1/\beta}$.  It follows 
 from   Lemma \ref{lem function10} that    $\theta< \varepsilon< 2\theta< 1$ for all  $n$ large enough. 
 Hence, by \eqref{eq psi'}  and the convexity of $\psi$,   we have
  \begin{align*}
 \psi(\varepsilon)\le\psi(\theta)+(\varepsilon-\theta)\psi'(2\theta)\le  \psi(\theta)-\frac{\varepsilon-
 \theta}{K_2(2\theta)^{1+1/\beta}}
 \le \psi(\theta)-\frac{L-c_{4,4}}{ 2^{1+1/\beta}K_2}.
 \end{align*} 
 This, together with \eqref{eq t n n+1}, implies that  for all $L$ large enough, 
 $$
 \varphi\left(\frac{\xi(t_{n-1})}{t_{n-1}^H}\right)\ge \varphi(\varepsilon)=\exp\big(-\psi(\varepsilon)\big)
 \ge 2 \exp\big(-\psi(\theta)\big)=2\varphi(\theta).
 $$ 
 The proof of Lemma \ref{lem finite}. is complete.
 \end{proof}

 \subsection{Necessity of the integral condition}
     
Suppose that with positive probability, $\xi(t)\le M(t)$ for all $t>0$ small enough. 
We will prove that $\xi(t)/t^H$ is bounded and $I_0(\xi)<+\infty$. 
The first fact is a direct consequence of \eqref{eq M1} and   
$\lim_{\theta\rightarrow \infty}\varphi(\theta)=1$.  
  Let us prove $I_0(\xi)<+\infty$ by using Lemma \ref{lem GBC}.
    
  \subsubsection{Construction of $\{t_n\}_{n\ge1}$}\label{subsect tn}
 \begin{lemma}\label{lem limit00} If   with positive probability, $\xi(t)\le M(t)$ for all $t>0$ small enough, then 
 \begin{align}\label{eq limit00}
 \lim_{t\rightarrow0}\frac{\xi(t)}{t^H}=0. 
 \end{align}
 \end{lemma}
 \begin{proof}
 Otherwise, we can find a sequence   $\{t_n\}_{n\ge1}$ such that $\delta:=\inf_{n\ge1}\frac{\xi(t_n)}{t_n^H}>0$.  
 Denote by $A_n:=\big\{M(t_{n})\le \xi(t_n)  \big\}$. By \eqref{eq M1}, we have 
 \begin{equation}\label{eq Pan}
 \inf_{n\ge1} \mathbb P(A_n)=\inf_{n\ge1} \mathbb P\big(M({t_n})\le \xi(t_n)\big)\ge \mathbb P\big(M(1)\le \delta\big)>0.
 \end{equation}
In the following,  we show that  the events $A_n$ occur infinitely often almost surely. This contradicts the assumption 
of Lemma \ref{lem limit00}.  Without loss of generality, we assume that $t_{n}/ t_{n+1}\ge 2$.  Denote by
$\mathbb P(A_n)=a_n$ for all $n\ge1$. Applying Proposition \ref{prop mtu1}   with $t=t_n, u=t_m,
 \theta=\xi(t_n)/t_n^H, \eta=\xi(t_m)/t_m^H$ for $m<n$, we have
 \begin{align*}
 \mathbb P\left(A_n\cap A_m\right)
 &\le \, a_na_m \Bigg\{\exp\left(-\frac{1}{K}\Big(\frac{t_m}{t_n}\Big)^{\tau}\right)+\varphi\left(\frac{\xi(t_m)}{t_m^H}\right)
 + \varphi\left(\frac{\xi(t_n)}{t_n^H}\right)\\
 &\quad \, \,\,\, \,\,\, \,\,\,+\exp\left[K\Big(\frac{t_m}{t_n}\Big)^{-\tau}  \Bigg(\Big(\frac{\xi(t_m)}{t_m^H}\Big)^{-1-1/\beta}
 + \Big(\frac{\xi(t_n)}{t_n^H}\Big)^{-1-1/\beta}\Bigg)  \right] \Bigg\}\\
 &\le a_na_m \Bigg\{\exp\left(-\frac{1}{K}\Big(\frac{t_m}{t_n}\Big)^{\tau} +2K_1\delta^{-1/\beta}   \right)
 +\exp\left(2K\Big(\frac{t_m}{t_n}\Big)^{-\tau} \delta^{-1-1/\beta}   \right) \Bigg\},
 \end{align*}
 where \eqref{eq sbp} is used in the last step.  For any $l\ge1, N\ge1$, let $\mathbb N_{l, N}:=\{n_k\}_{k\ge N}
 \subset \mathbb N$ such that for $t_{n_k}/t_{n_{k+1}}\ge l$ for any $n_{k},n_{k+1}\in \mathbb N_{l, N}$. 
 Thus,  we have
 \begin{align*}
 \mathbb P(A_n\cap A_m)\le&\, a_na_m \Bigg\{\exp\left(-\frac{1}{K} l^{\tau}+2K_1\delta^{-1/\beta} \right) 
 +  \exp\left(2Kl^{-\tau} \delta^{-1-1/\beta}   \right) \Bigg\}\\
 =:&\, a_na_m \left(1+o(l)\right),
 \end{align*}
 where $o(l)\rightarrow0$ as $l\rightarrow +\infty$. By \eqref{eq Pan}, we have $\sum_{m\in \mathbb N_{l, N}}
 a_m=+\infty$. Consequently,  by Lemma \ref{lem GBC}, we have 
 $$
 \mathbb P\bigg(\bigcup_{m\in \mathbb N_{l, N}} A_m\bigg)\ge \frac{1}{1+o(l)}.
 $$
 This implies that  $\lim_{N\rightarrow \infty}\mathbb P\left(\bigcup_{n\ge N} A_n\right)=1$. Hence,  
 the events $A_n$ occur infinitely often almost surely. The proof is complete.
 \end{proof}

\begin{lemma}\label{lem tn1}
Assume $I_0(\xi)=+\infty$ and  \eqref{eq limit00}. Then there exists  a sequence   $\{t_n\}_{n\ge1}$ 
with the following   properties:
 \begin{itemize}
   \item[(i)]  
  \begin{align}\label{eq sum infty0}
  \sum_{n=1}^{\infty} \varphi\left(\frac{\xi(t_n)}{t_n^H}\right)=+\infty;
  \end{align} 
 \item[(ii)]  
 \begin{align}\label{eq tn+10}
 t_{n+1}= t_n-t_{n+1}^{\gamma/\beta} \xi(t_{n+1})^{1/\beta}.
\end{align}
 
 \end{itemize}
 \end{lemma}
 
 \begin{proof} The construction is given by induction over $n$. We take $t_1=\e^{-\e}$. Having 
 constructed $t_n$,  we define  
 \begin{align*} 
 t_{n+1}:=\inf\left\{u\le t_n;\, u+u^{\gamma/\beta}\xi(u)^{1/\beta}\ge t_n \right\}.
 \end{align*}
It is obvious by the continuity of $\xi$ that 
\eqref{eq tn+10}   holds.   
 
  To prove \eqref{eq sum infty0},  it is sufficient to prove that for all $n$ large enough,  
 \begin{align}\label{eq In1}
 I_n:=\int_{t_{n+1}}^{t_{n}} \left(\frac{\xi(t)}{t^H}\right)^{-1/\beta}\varphi\left(\frac{\xi(t)}{t^H}\right) 
 \frac{dt}{t}\le c_{4,5}\varphi\left(\frac{\xi(t_{n})}{t_{n}^H}\right),
 \end{align}
 where  $c_{4,5}$  depends on $H$ and $\beta$ only.
By the monotonicities of $\{t_n\}_{n\ge1}$ and $\xi$, we derive
 \begin{align}\label{eq In2}
 I_n \le &\, \frac{\left(t_n-t_{n+1}  \right)}{  t_{n+1}^{\gamma/\beta}\xi(t_{n+1})^{1/\beta}}   
 \varphi\bigg(\frac{\xi(t_{n})}{t_{n+1}^H}\bigg)
 =\, \varphi\bigg(\frac{\xi(t_{n})}{t_{n+1}^H}\bigg).
 \end{align} 
It follows from \eqref{eq limit00} and \eqref{eq tn+10} that $t_{n+1}\ge  t_n/2$
  for all  $n$ large enough. Hence
 \begin{equation}\label{eq In3}
 \begin{split}
 \frac{\xi(t_{n})}{t_{n+1}^H}=&\,
 \frac{\xi(t_{n})}{t_{n}^H} \bigg(\frac{t_{n}}{t_{n+1}}\bigg)^H= \,\frac{\xi(t_{n})}{t_{n}^H}
 \Bigg(1+\bigg(\frac{\xi(t_{n+1})}{t_{n+1}^H}\bigg)^{1/\beta}\Bigg)^H\\
 \le &\,\frac{\xi(t_{n})}{t_{n}^H}\Bigg(1+ \bigg(\frac{2\xi(t_{n})}{ t_{n}^H}\bigg)^{1/\beta}\Bigg)^H
 \le  \,\frac{\xi(t_{n})}{t_{n}^H}\left(1+c_{4,6} \left(\frac{\xi(t_{n})}{ t_{n}^H}\right)^{1/\beta}\right),
 \end{split}
 \end{equation}
 for some constant $c_{4,6}>0$ which only depends on $H$.
 The inequality \eqref{eq In1}  follows now from \eqref{eq control}, \eqref{eq In2} and \eqref{eq In3}.  
 The proof is complete.  \end{proof}

 For each $n\ge 1$, we define $k(n)$ by 
 \begin{equation}\label{eq kn}
 2^{k(n)}\le \left(\frac{t_n^H}{\xi(t_n)}\right)^{1/\beta} < 2^{k(n)+1},
 \end{equation}
 and set $I_k:=\big\{n\in\mathbb N; \, k(n)=k \big\}$ for any $k\ge1$.  
  
Recall the constant $K_1$ given in \eqref{eq sbp}. Without loss of generality, we  assume 
that $K_1\ge c_{2,2}$ where $c_{2,2}$ is the constant in \eqref{eq Mtu}. For $k\ge1$, we set 
 $$N_k:=\exp\left(2^{k-2}/K_1 \right).$$

 The following is a refinement of Lemma \ref{lem tn1}.
 \begin{proposition}\label{Prop con2}  
Assume $I_0(\xi)=+\infty$, \eqref{eq limit00}, and  $\xi(t)/t^{(1+\varepsilon_0)H}$ is nonincreasing  
for some $\varepsilon_0>0$.  Then there exist a positive constant $c_{4,7}$ depending on $\beta$
and $H$ only and a set $J$ with the following properties:
 \begin{itemize}
 \item[(i)]
 \begin{align}\label{eq sum J infty0}
 \sum_{n\in J}\varphi\left( \frac{\xi(t_n)}{t_n^H}  \right)=+\infty;
 \end{align}
 
  \item[(ii)] Given $n, m\in J$ with $m<n$ such that 
  \begin{align}\label{eq kmn0}
  \card\big(I_{k(m)}\cap [m,n]\big)> N_{k(m)}, 
  \end{align}
 we have 
\begin{align}\label{eq tmn exp0}
\frac{t_m}{t_n}\ge\exp\left(\exp\big( 2^{\max\{k(m),\, k(n)\}}/ c_{4,7}  \big) \right).
\end{align} 
 \end{itemize}
 \end{proposition}
 \begin{proof} Set $a_n:=\varphi\big(\xi(t_n)/t_n^H\big)$. We  recall from \eqref{eq sbp} that 
 \begin{align*} 
 \exp\left(- K_1 2^{k(n)+1}\right)\le a_n \le 
 \exp\left(-   2^{k(n)}/K_1\right).
 \end{align*}
 Given $m, k\in \mathbb N$,   define
 $$
 U_{m,k}:=\big\{i>m; \, i\in I_k,\, \card(I_k\cap [m,i])\le N_k \big\}.
 $$
 Thus, 
 \begin{align*}
 \sum_{i\in U_{m, k}} a_i\le\, N_k \exp\left(-\frac{2^k}{K_1}\right)
 \le \,   \exp\left(-\frac{2^{k-1}}{K_1}\right). 
 \end{align*}
 Denote by $k_0$ the smallest integer such that $2^{k_0}\ge 2K_1+4K_1^2$. Then there exists 
 a constant $c_{4,8}\in (0,1)$ satisfying that 
 \begin{equation}\label{Eq:sums}
 \begin{split}
 \sum_{k\ge k(m)+k_0}\sum_{i\in U_{m,k}} a_i\le &\, \sum_{k\ge k(m)+k_0}\exp\left(-\frac{2^{k-1}}{K_1}\right)\\
 \le &\, a_m \sum_{k\ge k(m)+k_0}\exp\left(K_1 2^{k(m)+1}-\frac{2^{k-1}}{K_1}\right)\\
 \le &\, a_m \sum_{l\ge 0}\exp\left(  2^{k(m)}\Big(2K_1-\frac{2^{l-1+k_0}}{K_1}\Big) \right)\\
 \le &\, a_m \sum_{l\ge 0}\exp\left(  -2 ^{l} \right)\le c_{4,8} a_m.
 \end{split}
 \end{equation}
 For each $m\ge1$, set 
 $$
 V_m:=\bigcup_{k\ge k(m)+k_0}   U_{m,k}.
 $$
 It follows from (\ref{Eq:sums}) that
 \begin{align}\label{eq sum Vm0}
 \sum_{k\le p}\sum_{m\in I_k}\sum_{i\in V_m} a_i\le c_{4,8}
 \sum_{k\le p}\sum_{m\in I_k} a_m.
 \end{align}
Set 
 $$
 J_0:=\mathbb N\cap \left( \bigcup_{m\ge1} V_m\right)^c. 
 $$
 By the definition of $V_m$, we know that  $k(m)+k_0\le k(i)$ if $i\in V_m$. Thus, 
 $$
 \Big(\bigcup_{m\ge1} V_m\Big)\cap \Big(\bigcup_{k\le p}I_k\Big)\subset \bigcup_{k(m)\le p} V_m.
 $$
 This, together with \eqref{eq sum Vm0}, implies that 
 $$
 \sum_{i\notin J_0, k(i)\le p}  a_i\le c_{4,8}\sum_{k(i)\le p} a_i,
 $$
 and then
  $$
 \sum_{i\in J_0, k(i)\le p} a_i\ge (1-c_{4,8})\sum_{k(i)\le p} a_i.
 $$
 Letting $p\rightarrow\infty$, we obtain that 
 \begin{align}\label{eq sum J infty10}
 \sum_{n\in J_0}\varphi\left( \frac{\xi(t_n)}{t_n^H}  \right)=+\infty.
 \end{align} 
 
 Let $J:=J_0\cap[w,\infty) $ for a constant $w\in \mathbb N$, whose value will be chosen later.   
 Then \eqref{eq sum J infty10} implies \eqref{eq sum J infty0}. 
 
 Next, we prove (ii).
 For any $n,m\in J$ with $n<m$. If $i\in I_{k(m)}$,  then by \eqref{eq tn+10}  we have 
 \begin{align*}
 t_{i-1}\ge\, t_i\left(1+ \left(\frac{\xi(t_i)}{t_i^{H} }\right)^{1/\beta}\right)
\ge \, t_i\left(1+2^{-k(m)-1}\right).
 \end{align*}
Thus, when \eqref{eq kmn0} holds and 
when $w$ (hence  $k(m)$) is large enough,  we have 
\begin{equation}\label{eq tmn 10}
\begin{split}
\frac{t_m}{t_n}\ge \left(1+2^{-k(m)-1}\right)^{N_{k(m)}}\ge &\,\exp\left( 2^{-k(m)-2}N_{k(m)}\right)\\
\ge &\, \exp\left(\exp\left(2^{k(m)-3}/c_{4,9} \right)\right),
\end{split}
\end{equation}  
for some constant $c_{4,9}>0$.
This implies \eqref{eq tmn exp0} whenever $k(m)\ge k(n)-k_0$. If $k(m)<k(n)-k_0$, then by definition of $J$ 
we must have $n\notin U_{m, k(n)}$. This means that 
$$
\card\big(I_{k(n)}\cap [m,n]\big)>N_{k(n)},
$$
and the argument leading to \eqref{eq tmn 10} shows that \eqref{eq tmn 10} holds for $k(n)$ rather than $k(m)$.
 Thus, we complete the proof of this proposition  by choosing $w$ large enough. 
   \end{proof}

 \subsubsection{Proof of  the necessity in Theorem \ref{thm main 0} }\label{Sec Neces}
 Let $\{t_n\}_{n\in J}$ be as in  Proposition \ref{Prop con2}.
 Set $A_n= \big\{M({t_n})<\xi(t_n)\big\}$,  then $\mathbb P(A_n)=a_n= \varphi\big({\xi(t_n)}/{t_n^H}\big).$
 According to \eqref{eq sum J infty0} and  Lemma \ref{lem GBC}, it suffices to prove the following 
 statement: 
 \begin{quote}
 Given $\varepsilon>0$, there exist  positive constants $K$ and $q$ such that for any $n\in J$ with $n\ge q$, 
 \begin{align}\label{eq Anm00}
 \sum_{m\in J, \, m<n} \mathbb P(A_n\cap A_m)\le \mathbb P(A_n)\bigg(K+(1+\varepsilon) 
 \sum_{m\in J, \, m<n}\mathbb P(A_m) \bigg).
 \end{align}
 \end{quote}
 The sum on the left-hand side will be split into three parts over the following subsets of $J$:  for any   
 $n$, $k\in \mathbb N$, set 
 \begin{equation*}
 \begin{split}
 J'&:=\big\{m\in J; \ t_n<t_m\le 2 t_n \big\};\\
 J_k&:=\big\{m\in J\cap I_k; \, t_m>2t_n, \, \card(I_k\cap [m,n])\le N_k  \big\};\\
 J''&:=J\setminus \bigg( J' \cup \bigcup_{k\ge1} J_k \bigg).
 \end{split}
\end{equation*}
 
Applying  Eq. \eqref{eq Mtu}  with $t=t_n$, $u=t_m$,   $\theta=\xi(t_n)/t_n^{H}$ and $\eta =\xi(t_m)$ for 
any $m<n$, we have
\begin{align}\label{eq Anm10}
\mathbb P(A_n\cap A_m)\le  2a_n \exp\bigg(-\frac{t_m-t_n}{  c_{2,2} t_m^{\gamma/\beta}\xi(t_m)^{1/\beta}}\bigg).
\end{align}

(1). When $t_n< t_m\le 2 t_n$,   the monotonicity of $\xi(t)/t^{(1+\varepsilon_0)H}$  implies that  
$\xi(t_m)\le 2^{(1+\varepsilon_0)H}\xi(t_n)$. Hence, by \eqref{eq Anm10},  we have 
 \begin{align}\label{eq Anm200}
\mathbb P(A_n\cap A_m)\le 2 a_n \exp\bigg(-\frac{t_m-t_n}{ c_{2,2} 2^{\gamma/\beta} 2^{ (1+\varepsilon_0)H /\beta}   
t_n^{\gamma/\beta}\xi(t_n)^{1/\beta}}\bigg).
 \end{align}
By \eqref{eq tn+10} and the monotonicity of $\xi$,  we have  that for all  $m\le i < n$, 
$$
t_{i}= t_{i+1}+t_{i+1}^{\gamma/\beta}\xi(t_{i+1})^{1/\beta}\ge t_{i+1}+t_n^{\gamma/\beta}\xi(t_n)^{1/\beta}.
$$
 This, together with \eqref{eq Anm200}, implies that 
 $$
\mathbb P(A_n\cap A_m)\le  2a_n \exp\left(-\frac{m-n}{ c_{2,2} 2^{\gamma/\beta}2^{(1+\varepsilon_0)H/\beta}   }\right).
 $$
Therefore, we have 
 \begin{align}\label{eq J'0}
 \sum_{m\in J'}\mathbb P(A_n\cap A_m)\le c_{4,10}a_n,
 \end{align}
for some constant $c_{4,10}>0$.
  
 (2).  When $t_m > 2 t_n$,  by \eqref{eq Anm10}  we know  
  \begin{align*} 
\mathbb P(A_n\cap A_m)\le 2 a_n \exp\bigg(-\frac{t_m^{H/\beta}}{ 2 c_{2,2}  \xi(t_m)^{1/\beta}}\bigg).
 \end{align*}
  Thus,  we have that for any $m\in I_k$,
  $$
\mathbb P(A_n\cap A_m)\le  2a_n\exp\left(-\frac{2^{k-1}}{c_{2,2}}\right).
  $$
  Since $J_k\subset I_k$ and $\card(J_k)\le N_k$, using the definition of $N_k$, we have 
  \begin{align*}
  \sum_{m\in J_k} \mathbb P(A_n\cap A_m)\le 2 N_k a_n\exp\left(-\frac{2^{k-1}}{c_{2,2}} \right) 
  \le 2 a_n \exp\left(-\frac{2^{k-2}}{K_1} \right).
  \end{align*}
  Thus, 
   \begin{align}\label{eq Anm40}
\sum_{m\in \cup_{k\ge1} J_k }\mathbb P(A_n\cap A_m)\le  c_{4,11}a_n,
 \end{align}
   for some constant $c_{4,11}>0$.

 (3). By using \eqref{eq J'0} and \eqref{eq Anm40}, in  order to prove \eqref{eq Anm00}, it suffices to 
 prove the following result:   
 \begin{quote}
 Given $\varepsilon>0$, there exists  a constant $c_{4,12}>0$ such that for any
 $$n>\sup_{l\le c_{4,12}}
 \sup \{k;k\in I_l\}  \ \ \text{and } m\in J''\ \  \text{with  }m<n,$$
  we have 
  \begin{align}\label{eq Anm50}
  \mathbb P(A_n\cap A_m)\le  a_na_m(1+\varepsilon).
  \end{align}
  \end{quote}
  
Assume $m\in J''$  such that $m<n$.  
 Applying    \eqref{eq mtu1} with $t=t_n$, $u=t_m$, $\theta=\xi(t_n)/t_n^H$ and $\eta= \xi(t_m)/t_m^H$,
  and using the facts of $\theta\ge 2^{-(k(n)+1)\beta}$ and $\eta\ge 2^{-(k(m)+1)\beta}$,    we have
    \begin{equation}\label{eq Anm20}
  \begin{split}
  \mathbb P\left( A_n\cap A_m \right)
  \le &\,  a_na_m \Bigg\{ \exp\left[-\frac{1}{K} \left(\frac{t_m}{t_n}\right)^{\tau}+\psi(\theta)+\psi(\eta)  \right]\\
  &\, \, \, \, \,  \,\, +\exp\left[K\left(\frac{t_m}{t_n}\right)^{-\tau} \left( 2^{(k(n)+1)(1+1/\beta)} +2^{(k(m)+1)
  (1+1/\beta)} \right)\right]  \Bigg\}.
  \end{split}
  \end{equation}
   Since $m\in J''$, $m\notin J_{k(m)}$. By the definition of $J_{k(m)}$, we have 
$$
\card\big(I_{k(m)}\cap [m, n]\big)>N_{k(m)}.
$$ 
Thus, \eqref{eq kmn0} holds. By \eqref{eq tmn exp0}, we have 
\begin{align}\label{eq tmn30}
\frac{t_m}{t_n}\ge\exp\left(\exp\big(2^{\max\{k(m),\, k(n)\}}/c_{4,7} \big) \right).
\end{align}
By \eqref{eq sbp} and \eqref{eq kn},    $\psi(\theta)\le K_1 2^{k(n)+1}$ and $\psi(\eta)\le K_12^{k(m)+1}$. 
Those,  together with  \eqref{eq tmn30}, imply that the coefficient of $a_na_m$ in \eqref{eq Anm20} gets  
close to $1$ as $\max\left\{k(m),\, k(n)\right\}$   becomes large. Hence, we get \eqref{eq Anm50}.     
The proof of   the necessity in   Theorem \ref{thm main 0}  is complete.
 \vskip0.5cm

\section{Proof of Theorem \ref{thm main infty}}
In this section, we prove an integral criterion for the lower classes of GFBM at  infinity. The setting is the 
same as in Talagrand \cite{Tal96} and the arguments are similar to those in \cite{Tal96} or Section 4 
of this paper. In order not to make the paper too lengthy, we only give a sketch of the proof.

\subsection{Sufficiency}

Suppose that $\xi(t)$ is a nondecreasing continuous function such that $\xi(t)/t^H$ is bounded and 
$I_{\infty}(\xi)<+\infty$. We  prove that $\xi(t)\le M(t)$ for $t$ large enough in probability one.

By using the argument in the proof of \cite[Lemma 3.1]{Tal96}, one can obtain the following analogue of
Lemma 4.1. The details of the proof are omitted here.
\begin{lemma}\label{lem function0} Suppose that $\xi(t)$ is a nondecreasing continuous function such 
that $\xi(t)/t^H$ is bounded and $I_{\infty}(\xi)<+\infty$. Then
\begin{align}\label{eq f0}
\lim_{t\rightarrow\infty}\frac{\xi(t)}{t^H}=0.
\end{align}
\end{lemma}

In order to prove the sufficiency, we  construct the sequences $\{t_n\}_{n\ge1}$,  $\{u_n\}_{n\ge1}$ 
and $\{v_n\}_{n\ge1}$ recursively as follows. Let $L>2H$ be a constant. We start with 
$t_1=\e^{\e}$. Having constructed $t_n$, we set
\begin{align} 
u_{n+1}:=&\, t_n+t_n^{\frac{\gamma}{\beta}}\xi(t_n)^{\frac1\beta},  \label{eq tn infty0} \\
v_{n+1}:=&\, \inf\left\{u>t_n; \,\xi(u)\ge\xi(t_n)\left( 1+L\left(\frac{\xi(t_n)}{t_n^H} \right)^{1/\beta}\right) 
\right\}, \label{eq tn infty1} \\
t_{n+1}:=&\, \min\{u_{n+1},\, v_{n+1}\}. \label{eq tn infty2}
\end{align}

Similar to the proofs of \cite[Lemmas 3.2 and 3.3]{Tal96}, we have the following lemmas. 
\begin{lemma}\label{lem infty} $\{t_n\}_{n\ge1}$ is increasing and
$\lim_{n\rightarrow\infty} t_n=+\infty.$
\end{lemma}

\begin{lemma}\label{lem Mtn}
If $M(t_n)\ge \xi(t_n)\left(1+L\left(\frac{\xi(t_n) }{t_n^H}\right)^{1/\beta}\right)$ for all $n\ge n_0$, then 
$M(t)\ge \xi(t)$ for all $t\ge t_{n_0}$.
\end{lemma}

By \eqref{eq M1}, we have 
\begin{align}
\mathbb P\Bigg(M(t_n)<\xi(t_n)\bigg(1+L\Big(\frac{\xi(t_n)}{t_n^H}\Big)^{1/\beta} \bigg) \Bigg)
=\varphi\Bigg(\frac{\xi(t_n)}{t_n^H} \bigg(1+L\Big(\frac{\xi(t_n)}{t_n^H}\Big)^{1/\beta}\bigg) \Bigg).
\end{align}
One could show that this later series converges by using the same argument in    \cite[Section 3]{Tal96} or     
\cite[Section 3]{ElN2011}. Therefore, the proof of the sufficiency is finished   by the  
Borel-Cantelli lemma,  Lemma \ref{lem infty}  
and Lemma \ref{lem Mtn}.

 \subsection{Necessity }
 
In this part,  we suppose that with positive probability, $\xi(t)\le M(t)$ for all $t$ large enough. 
We prove that $\xi(t)$ is bounded and $I_{\infty}(\xi)<+\infty$. The first fact is a direct 
consequence of \eqref{eq M1} and the fact that $\lim_{\theta\rightarrow \infty}\varphi(\theta)=1$. 
To prove the second statement, we use Lemma \ref{lem GBC} and the following lemmas in a 
way similar to the proof in Section 4.2.

 \begin{lemma}\label{lem limit0} Suppose that with positive probability  $\xi(t)\le M(t)$ for all $t$ large enough. 
 Then we have 
 \begin{align}\label{eq limit0}
 \lim_{t\rightarrow\infty}\frac{\xi(t)}{t^H}=0. 
 \end{align}
 \end{lemma}
 \begin{proof} Since its proof is similar to that of \cite[Lemma 4.1]{Tal96}, we omit the details here.
 \end{proof}

 To prove the necessity, we will show that $\xi\in LUC_{\infty}(M)$ when  
\begin{align}\label{Eq Nece} I_{\infty}(\xi)=+\infty\  \ \text{and }\, \ \lim_{t\rightarrow\infty} \frac{\xi(t)}{t^H}=0.
\end{align} 
The first step is to construct a suitable sequence. 
\begin{lemma}\label{lem tninfty}  Under    \eqref{Eq Nece},  there exists a sequence 
$\{t_n\}_{n\ge1}$ with the following three properties:
 \begin{itemize}
 \item[(i)]  
 \begin{align}\label{eq tn+1}
 t_{n+1}\ge t_n\left(1+\left(\frac{\xi(t_n)}{t_n^H} \right)^{1/\beta}\right);
\end{align}

 \item[(ii)] 
  \begin{align}\label{eq sum infty}
  \sum_{n=1}^{\infty} \varphi\left(\frac{\xi(t_n)}{t_n^H}\right)=+\infty;
  \end{align}

 \item[(iii)] For all  $m\ge n$ large enough, 
 \begin{align}\label{eq tmn}
 \frac{\xi(t_m)}{t_m^{ 2H}}\le   2  \frac{\xi(t_n)}{t_n^{2H}}.
 \end{align} 
 \end{itemize}
  
 \end{lemma}
 
 \begin{proof} The construction is given by induction over $n$. We take $t_1=\e^{\e}$. Having 
 constructed $t_n$, we find $s_n\ge t_n$ such that 
 $$
 \sup\left\{\frac{\xi (t)}{t^{2H}};\, t\ge t_n\right\}=\frac{\xi(s_n)}{s_n^{2H}}.
 $$
 Using  the continuity of $\xi$ and \eqref{eq limit0}, we know $s_n\in (t_n,\infty)$. We then set 
 \begin{align}\label{eq tn+11}
 t_{n+1}:=s_n+s_n^{\gamma/\beta} \xi(s_n) ^{1/\beta}=s_n\left(1+\left(\frac{\xi(s_n)}{s_n^H}\right)^{1/\beta}\right).
 \end{align}
It is obvious   that 
\eqref{eq tn+1} holds.  By the construction of $s_{n-1}$ and by \eqref{eq limit0}, we have  that for $m\ge n$,
$$
\frac{\xi(t_m)}{t_m^{2H}}\le 
\frac{\xi(s_{n-1})}{s_{n-1}^{2H}}.
$$
 By \eqref{eq limit0} and \eqref{eq tn+11}, we know for all $n$ large enough, $\xi(s_{n-1})\le \xi (t_n)$ and 
 $t_n^{2H}\le 2 s_{n-1}^{2H}$.  Thus, \eqref{eq tmn} holds. To prove \eqref{eq sum infty}, we   show that 
 for all $n$ large enough, 
 \begin{align}\label{eq In}
 I_n:=\int_{t_n}^{t_{n+1}} \left(\frac{\xi(t)}{t^H}\right)^{-1/\beta}\varphi\left(\frac{\xi(t)}{t^H}\right) \frac{dt}{t}
 \le c_{5,1}\varphi\bigg(\frac{\xi(t_{n+1})}{t_{n+1}^H}\bigg),
 \end{align}
where  the constant $c_{5,1}$  depends on $H$ and $\beta$ only.
 
 We write 
 \begin{align}\label{eq: In12}
 I_n=\left[\int_{t_n}^{s_{n}}+\int_{s_n}^{t_{n+1}} \right] \left(\frac{\xi(t)}{t^H}\right)^{-1/\beta}\varphi
 \left(\frac{\xi(t)}{t^H}\right) \frac{dt}{t}=:I_n^1+I_n^2.
 \end{align}
 First, we have 
 \begin{align}\label{eq: In121}
 I_n^2= &\,\int_{s_n}^{t_{n+1}} \left(\frac{\xi(t)}{t^H}\right)^{-1/\beta}\varphi\left(\frac{\xi(t)}{t^H}\right) \frac{dt}{t}\\
 \le &\,  \left(t_{n+1}-s_n \right)s_n^{-\gamma/\beta}  \xi(s_n)^{-1/\beta}  \varphi\left(\frac{\xi(t_{n+1})}{s_n^H}\right)\\
 =&\, \varphi\bigg(\frac{\xi(t_{n+1})}{s_n^H}\bigg).
 \end{align}
 By \eqref{eq tn+11}, we know that  $t_{n+1}\le 2 s_n$ for all $n$ large enough. Hence, 
 \begin{equation}\label{eq: In122}
 \begin{split}
 \frac{\xi(t_{n+1})}{s_n^H}= \,
 \frac{\xi(t_{n+1})}{t_{n+1}^H}\bigg(\frac{t_{n+1}}{s_n}\bigg)^H 
 =& \,\frac{\xi(t_{n+1})}{t_{n+1}^H}\Bigg(1+\bigg(\frac{\xi(s_n)}{s_n^H}\bigg)^{1/\beta}\Bigg)^H\\
 \le &\,\frac{\xi(t_{n+1})}{t_{n+1}^H}\Bigg(1+2^{H/\beta}\bigg(\frac{\xi(t_{n+1})}{t_{n+1}^H}\bigg)^{1/\beta}\Bigg)^H\\
 \le & \,\frac{\xi(t_{n+1})}{t_{n+1}^H}\Bigg(1+c_{5,2}\bigg(\frac{\xi(t_{n+1})}{t_{n+1}^H}\bigg)^{1/\beta}\Bigg),
 \end{split}
 \end{equation}
 here $c_{5,2}$ is a positive constant which only depends on $H$ and $\beta$. 
 
 By \eqref{eq control},  \eqref{eq: In121}  and \eqref{eq: In122}, we know that there exists a constant 
 $c_{5,3}>0$ satisfying that 
 $$
 I_n^2\le c_{5,3} \varphi\bigg(\frac{\xi(t_{n+1})}{t_{n+1}^H}\bigg). 
 $$ 
 
 Now, we turn to study $I_n^1$.  By the construction of $s_n$, we have for any $t_n\le t\le s_n$, 
 $$
 \frac{\xi(t)}{t^H}\le t^H
 \frac{\xi(s_n)}{s_n^{2H}}\le 
 \frac{\xi(s_n)}{s_n^H}. 
 $$
By \eqref{eq limit0}, we can assume $\xi(s_n)/ s_n^H$ is arbitrarily small for all large $n$, so that 
Lemma \ref{lem increase} implies that for any $t_n\le t\le s_n$, 
 $$
 \left(\frac{\xi(t)}{t^H}\right)^{-1/\beta}\varphi \left(\frac{\xi(t)}{t^H}\right)\le \left( \frac{t^H\xi(s_n)}
 {s_n^{2H}}\right)^{-1/\beta} \varphi \left( \frac{t^H\xi(s_n)}{s_n^{2H}}\right),
 $$
 and thus
 $$
 I_n^1\le \int_{t_n}^{s_n}\left(\frac{\xi(t)}{t^H} \right)^{-1/\beta}\varphi \left(\frac{\xi(t)}{t^H}\right)\frac{dt}{t}
 \le  \int_{t_n}^{s_n} \left( \frac{t^H\xi(s_n)}{s_n^{2H}}\right)^{-1/\beta} \varphi \left( \frac{t^H\xi(s_n)}
 {s_n^{2H}}\right)\frac{dt}{t}.
 $$
 Making the change of variable $t=us_n$, we have
 $$
 I_n^1\le \int_0^1 \left(u^H a\right)^{-1/\beta}\varphi\left(u^H a\right)\frac{du}{u},
 $$
 where $a= \xi(s_n)/s_n^H$. It remains to prove that this later integral is at most $c_{5,4}\varphi(a)$ 
 for some positive constant $c_{5,4}$.  By \eqref{eq varphi1}, it suffices to prove that for any $c_{5,5}
 \in (0,1)$, there exists $c_{5,6}>0$ such that
 $$
 \int_{c_{5,5}}^1 \left(u^H a\right)^{-1/\beta}\varphi\left(u^H a\right)\frac{du}{u} \le c_{5,6}\varphi(a).
 $$ 
 
 Setting $v=u^H$, it suffices to prove that 
\begin{align}\label{eq varphi 01}
 \int_0^1 \varphi(va)dv=\int_0^1\varphi\big((1-v)a\big)dv\le  c_{5,6} a^{1/\beta} \varphi(a).
 \end{align}
 By the convexity of $\psi$ and \eqref{eq psi'}, we have
 \begin{align*}
 \psi(a-av)\ge \, \psi(a)+av|\psi'(a)|
 \ge\, \psi(a)+\frac{av}{K_2a^{1+\beta}}.
 \end{align*}
This implies that
 $$
 \varphi\big((1-v)a\big)\le \varphi(a)\exp\left(-\frac{v}{K_2 a^{1/\beta}}\right).
 $$
 Since 
 $$
 \int_0^1 \exp\left(-\frac{v}{K_2 a^{1/\beta}}\right)dv\le K_2 a^{1/\beta},
 $$
Eq. \eqref{eq varphi 01} is proved and the proof is complete. 
 \end{proof}

 For each $n\ge1$, we define $k(n)$ by 
 $$
 2^{k(n)}\le \left(\frac{t_n^H}{\xi(t_n)}\right)^{1/\beta} < 2^{k(n)+1},
 $$
 and we set $I_k:=\big\{n\in\mathbb N; \, k(n)=k \big\}$ for any $k\ge1$. By \eqref{eq limit0}, we know that each $I_k$ is finite.
 
 We recall the constant $K_1$ given in \eqref{eq sbp}. Without loss of generality, we can assume that 
 $K_1\ge c_{2,2}$,  where $c_{2,2}$ is the constant in \eqref{eq Mtu}. For any $k\ge1$, we set 
 $N_k:=\exp\left(2^{k-2}/K_1 \right)$.

 By using the argument in \cite[Proposition 4.2]{Tal96} or Proposition  \ref{Prop con2}, we  obtain the following result. 
 \begin{lemma}\label{Prop con22}  Under    \eqref{Eq Nece}, there exist a positive constant $c_{5,7}$ depending on 
 $\beta$ and $H$ only and a set $J$ with the following properties:
 \begin{itemize}
 \item[(i)]
 \begin{align}\label{eq sum J infty}
 \sum_{n\in J}\varphi\left( \frac{\xi(t_n)}{t_n^H}  \right)=+\infty;
 \end{align}
 
\item[(ii)] Given $n, m\in J$ with $n<m$ such that 
\begin{align}\label{eq kmn}
\card\big(I_{k(m)}\cap [n,m]\big)> N_{k(m)}, 
  \end{align}
 we have 
\begin{align}\label{eq tmn exp}
\frac{t_m}{t_n}\ge\exp\bigg(\exp\Big( 2^{\max\{k(n),\, k(m)\}}/ c_{5,7}  \Big) \bigg).
\end{align} 
 \end{itemize}
 \end{lemma}
  
By using Lemmas \ref{lem tninfty} and \ref{Prop con22} and using the argument in  
Section \ref{Sec Neces} (also see \cite[Section 5]{Tal96}), one can prove the necessity part of 
Theorem \ref{thm main infty}. The details are omitted.

 \vskip0.5cm
\noindent{\bf Acknowledgments}:  The research of  R. Wang  is partially supported by  NNSFC grant 11871382  
and the Fundamental Research Funds for the Central Universities 2042020kf0031. The research of Y. Xiao is 
partially supported by NSF grant DMS-1855185.

\medskip 

\bigskip

\end{document}